\documentclass[onefignum,onetabnum]{siamonline220329}

\usepackage{bm}
\usepackage{booktabs}
\usepackage{mathtools}
\usepackage{subcaption}
\usepackage{tikz}
\usepackage{amsfonts}
\usepackage{algorithm}
\usepackage{algpseudocode}

\allowdisplaybreaks

\DeclareMathOperator{\diag}{diag}

\DeclareMathOperator*{\id}{id}

\DeclareMathOperator*{\argmin}{argmin}

\DeclareMathOperator{\GW}{GW}
\DeclareMathOperator{\FGW}{FGW}

\DeclareMathOperator{\UMGW}{UMGW}
\DeclareMathOperator{\UMFGW}{UMFGW}
\DeclareMathOperator{\UFGW}{UFGW}
\DeclareMathOperator{\MFGW}{MFGW}
\DeclareMathOperator{\UGW}{UGW}
\DeclareMathOperator{\MGW}{MGW}

\DeclareMathOperator{\KL}{KL}
\DeclareMathOperator{\TV}{TV}

\DeclareMathOperator{\F}{F}

\newcommand{\vX}{\bm{X}}
\newcommand{\vXA}{\bm{X\!\!\!A}}
\newcommand{\vA}{\bm{A}}
\newcommand{\vx}{\bm{x}}
\newcommand{\va}{\bm{a}}
\newcommand{\vy}{\bm{y}}

\newcommand{\mpi}{\bm{\pi}}

\newcommand{\weakly}{\rightharpoonup}
\newcommand{\N}{\ensuremath{\mathbb{N}}}
\newcommand{\M}{\mathcal{M}}

\newcommand{\R}{\ensuremath{\mathbb{R}}}

\newcommand{\XX}{\mathbb{X}}
\newcommand{\YY}{\mathbb{Y}}
\newcommand{\Xf}{\mathfrak{X}}
\newcommand{\Yf}{\mathfrak{Y}}

\newcommand{\dx}{\,\mathrm{d}}
\newcommand{\tT}{\mathrm{T}}
\newcommand{\p}{\mathcal{P}}
\newcommand{\bal}{\mathrm{bal}}
\newcommand{\free}{\mathrm{free}}
\newcommand{\ent}{\mathrm{BS}}
\newcommand{\struc}{\mathrm{mm}}
\newcommand{\lab}{\mathrm{lb}}
\newcommand{\fused}{\mathrm{fused}}
\newcommand{\bi}{\mathrm{bi}}

\newcommand{\eps}{\varepsilon}

\definecolor{PLOTblue}{rgb}{0.122,0.467,0.706}
\definecolor{PLOTorange}{rgb}{1,0.498,0.055}
\definecolor{PLOTgreen}{rgb}{0.173,0.627,0.173}

\newlength{\labwidth}

\newsiamremark{remark}{Remark}
\newsiamremark{example}{Example}
\newsiamremark{hypothesis}{Hypothesis}
\crefname{hypothesis}{Hypothesis}{Hypotheses}
\newsiamthm{claim}{Claim}

\headers{Multi-marginal Gromov--Wasserstein Transport and Barycenters}
{F. Beier, R. Beinert, and G. Steidl}

\title{Multi-marginal Gromov--Wasserstein Transport and Barycenters%
  \thanks{This work was funded by
      the German Research Foundation (DFG)
      within the RTG~2433 DAEDALUS
      and by the BMBF project ``VI-Screen'' (13N15754).}}

\author{Florian Beier%
  \thanks{Institute of Mathematics,
    Technische Universit\"at Berlin,
    Stra\ss{}e des 17. Juni 136,
    10623 Berlin, Germany 
    (\email{beier@math.tu-berlin.de},
    \email{beinert@math.tu-berlin.de},
    \email{steidl@math.tu-berlin.de}).}
  \and Robert Beinert\footnotemark[2]
  \and Gabriele Steidl\footnotemark[2]}

\begin{document}

\maketitle

\begin{abstract}
Gromov--Wasserstein (GW) distances 
are combinations of Gromov--Hausdorff and 
Wasserstein distances
that allow the comparison of two different metric measure spaces (mm-spaces).
Due to their invariance under measure- and 
distance-preserving transformations,
they are well suited 
for many applications in graph and shape analysis.
In this paper,
we introduce the concept of multi-marginal GW transport 
between a set of mm-spaces
as well as its
regularized and unbalanced versions. 
As a special case,
we discuss multi-marginal fused variants,
which combine the structure information of an mm-space with label information from an additional label space.
To tackle the new formulations numerically,
we consider the bi-convex relaxation of the multi-marginal GW problem, which is tight in the balanced case
if the cost function is conditionally negative definite.
The relaxed model can be solved by an alternating minimization,
where each step can be performed by a multi-marginal Sinkhorn scheme.
We show relations of our multi-marginal GW problem  
to (unbalanced, fused) GW barycenters
and present various numerical results,
which indicate the potential of the concept.
\end{abstract}



% ---------------------------------
\section{Introduction}
% ---------------------------------
Optimal transport (OT) seeks to optimally transport mass between two input measures according to some underlying cost function. Due to various applications, 
OT and its regularized and unbalanced variants
have attracted increasing attention both from the theoretical and computational point 
of view in recent years \cite{PC19book,S15otapplied}. The metric version of OT, referred to as the Wasserstein distance, is of particular interest \cite{villani2008optimal,ABS21}.
The exact matching of the input probabilities can be relaxed, which gives rise to 
several distances between arbitrary measures using unbalanced OT \cite{LMS18,CPSV18}. Regularized OT, which is numerically much more appealing, leads to the so-called Sinkhorn distances, which may be computed by the parallelizable Sinkhorn algorithm \cite{C2013}. Associated computational aspects are treated in \cite{CPSV17,SFVP19}.
For certain practical tasks as matching for teams \cite{CE10}, 
particle tracking \cite{CK18} and information fusion \cite{DC14,EHJK20}, 
it is useful  to compute transport plans between more than two marginal measures. 
This is possible in the framework of multi-marginal OT \cite{GS98,Pass15}, which was tackled numerically for Coulomb costs in \cite{BCN16} 
and for repulsive costs in \cite{CDD15,GKR20}.
Regularized and unbalanced variants of multi-marginal OT, which
can be efficiently solved for tree-structured costs 
using Sinkhorn iterations, were treated in \cite{BLNS2021,flows19BCN,tree21HRCK,LHXCJ20fastIBP}.
Furthermore, the multi-marginal OT with special cost function 
is closely related to the OT barycenter problem, see \cite{AB20,ABM16brute,JvL2022}.
The latter essentially constitute Fréchet means with respect to the OT divergence. Barycenters are of particular interest in the Wasserstein case and have been successfully applied for texture mixing \cite{texturemix11} and Bayesian inference \cite{bayes18}.
In an effort to reduce computational costs, slicing strategies can been leveraged to approximate (multi-marginal) OT \cite{sliced_gw,cohen2021sliced}.

For some applications such as shape matching, 
a drawback of OT lies in the fact that it is not invariant to 
isometric transformations such as translations or rotations.
A remedy are Gromov--Wasserstein (GW) distances, 
which were first considered by M\'emoli in \cite{memoli2011gromov}
as a modification of Gromov--Hausdorff distances and Wasserstein distances. 
In this framework, the objective is to match input measures living on maybe different metric spaces so that pairwise intrinsic distances are preserved.
For a survey of the geometry of Gromov--Wasserstein distances, we refer to \cite{sturm2020space}. The OT literature inspired a number of generalizations for GW such as unbalanced GW distances, which were studied in \cite{SejViaPey21},
a sliced version of GW distances in \cite{sliced_gw,BHS21} and a linear one in \cite{beier2021linear}. 
A fused Gromov--Wasserstein distance, which also accounts for labelled input data, was introduced in \cite{vayer2020fused}.
A linear fused version has been discussed in \cite{nguyen2022linearfused} and a spherical sliced one in \cite{nguyen2020improving}. An unbalanced fused Gromov--Wasserstein distance has been used to align cortical surfaces in \cite{fugw}.
Furthermore, Gromov--Wasserstein distances were examined for Gaussian measures
in \cite{le2021entropic,salmona2021gromov}.
A generalization of GW distances, called co-optimal transport (COOT), was introduced in \cite{titouan2020co}, 
and its unbalanced version was discussed in \cite{tran2022unbalanced}. Finally, the framework of optimal tensor transport \cite{ott} provides a general formulation which encompasses OT, GW and COOT.

Similarly to the OT case, GW barycenters have gained increasing attention since their first study in \cite{PCS2016}. 
Besides the computation of the measure,
the GW barycenter problem also consists in 
the determination of the underlying metric space.
To obtain a numerically tractable algorithm,
Peyré, Cuturi, and Solomon \cite{PCS2016} employ an entropic regularization
of the GW distance
and fix the number of support points and the measure of the barycenter beforehand. 
In other words, 
the objective is here only minimized with respect to the internal geometry
of the metric space.
In \cite{s-gwl,XLZD2019}, a similar procedure is used, where the entropic regularizer is replaced by a proximal term.
Moreover, Brogat-Motte et al.\ \cite{brogat2022learning} train a neural network to estimate fused GW barycenters with an a priori fixed discrete uniform distribution.
A disadvantage of these methods is 
that the GW barycenter of two images (interpreted as probability measures)
is a generic metric measure space.
For applications in imaging,
it would therefore be desirable to fix the geometric structure
and to solely minimize with respect to the unknown measure.

The existing literature about GW distances and its applications
mainly studies GW transports and matchings between two input measures.
In certain applications like
clustering of shapes \cite{PCS2016,beier2021linear}, 
multi-graph partitioning and matching \cite{s-gwl}, 
and graph prediction \cite{brogat2022learning},
a simultaneous GW-like matching of an arbitrary number of inputs is needed.
So far this problem has been resolved by utilizing a reference input such as GW barycenters and using bi-marginal GW transports.
An alternative approach to deal with the simultaneous matching
would be to generalize the multi-marginal OT variants with respect to the GW setting.

\paragraph{Contributions}
In this paper, inspired by multi-marginal OT, we propose a multi-marginal GW  formulation which directly allows for a simultaneous matching of an arbitrary number of inputs.
We also treat the associated regularized, unbalanced and fused variants.
Secondly, we study the associated bi-convex relaxation and show a tightness result in the balanced setting.
Thirdly, we leverage this new formulation to obtain two results regarding GW barycenters. One is mainly of theoretical nature and shows that barycenters are completely characterised by (multi-marginal) GW plans. The second result shows an unbalanced multi-marginal characterisation for barycenters with fixed support. We use the latter to obtain a novel computationally tractable algorithm for GW barycenters
with possible applications in imaging and data processing.

\paragraph{Outline of the paper}
In \cref{sec:gromov-wasserstein}, we recall basic facts on GW distances. Then, in \cref{sec:UMGW}, we introduce 
the concept of multi-marginal GW transport. A tight relaxation of the multi-marginal GW transport
is considered in \cref{sec:bi-convex-relax}. In particular, we will
make use of marginal conditionally negative semi-definite kernels.
The relation of multi-marginal GW transport
to the GW barycenter problem is examined in \cref{sec:bary}.
Multi-marginal fused GW transports and barycenters are studied in \cref{sec:fused}.
Numerical proof-of-concept results in imaging are provided in \cref{sec:numerics}.

%---------------------------------------------------------------------------
\section{Gromov--Wasserstein Distance} \label{sec:gromov-wasserstein}
%---------------------------------------------------------------------------
A \emph{metric measure space} (mm-space) is a triple $\XX \coloneqq (X,d,\mu)$
consisting of a compact metric space $(X,d)$
and a Borel probability measure $\mu$
on the Borel $\sigma$-algebra
induced by the metric $d$ on $X$. 
Within the paper, we stick to this compact setting, although a more general treatment is possible, see e.g.\ \cite{sturm2020space}.
In the following, we denote by $\mathcal M(X)$ the Banach space of signed Borel measures equipped 
with the total variation norm 
$\|\cdot\|_{\TV}$,
by $\mathcal M^+(X)$ the subset of positive Borel measures
and by 
$\p(X)$ the subset of probability measures. 
Recall that a sequence $(\mu_n)_{n \in \N} \subset \mathcal M(X)$ \emph{converges weakly} to $\mu \in  \mathcal M(X)$,
written $\mu_n \weakly \mu$,
if 
$\int_X \varphi \, \dx \mu_n(x) \to \int_X \varphi \, \dx \mu(x)$
for all continuous functions $\varphi$ on $X$.

For the mm-spaces
$\XX_1 \coloneqq (X_1,d_1,\mu_1)$ and $\XX_2 \coloneqq (X_2,d_2,\mu_2)$,
the \emph{Gromov--Wasserstein distance} is defined by
\begin{equation} \label{GW}
    \GW(\XX_1,\XX_2) 
    \coloneqq \inf_{\pi \in \Pi(\mu_1,\mu_2)} \Big( 
    \int_{(X_1 \times X_2)^2}
    \lvert d_1(x_1,x_1') - d_2(x_2,x_2')\rvert^2
    \dx \pi(x_1,x_2) \dx \pi(x_1',x_2') \Big)^\frac12,
\end{equation}
where $\Pi(\mu_1,\mu_2)\subset \p(X_1 \times X_2)$
consists of all measures with marginals $\mu_1$ and $\mu_2$.
By \cite[Cor~10.1]{memoli2011gromov},
the infimum in \cref{GW} is attained by some 
optimal GW transport plan $\pi^*$, which is in general not unique.
Furthermore,
$\GW(\XX_1,\XX_2) = 0$
if and only if
there exists a measure-preserving isometry, i.e.\ an isometry $\mathcal I \colon (X_1, d_1) \to (X_2, d_2)$
with $\mu_2 = \mathcal I_\# \mu_1 \coloneqq \mu_1 \circ \mathcal I^{-1}$.
In this case, we call $\XX_1$ and $\XX_2$ \emph{isomorphic} to each other.
The GW distance defines a metric
on the quotient space of mm-spaces
with respect to measure-preserving isometries \cite[Thm~5.1]{memoli2011gromov}.

Instead of fixing the marginals of the transport plan,
Séjourné, Vialard, and Peyré \cite{SejViaPey21} transferred
the concept of unbalanced OT \cite{LMS18}
to the GW setting by penalizing
the divergences between the marginals and the given measures.
To this end, recall that, for a 
so-called \textit{entropy function},
i.e.\ a convex and lower semi-continuous functions $\phi \colon \R_{\ge 0} \to [0,\infty]$
satisfying $\phi(1) = 0$, and 
two measures $\mu,\nu \in \mathcal M^+(X)$
with Radon--Nikodým decomposition $\mu = \dx \mu / \dx \nu + \mu^\perp$, 
the \emph{(Csiszár) $\phi$-divergence} is given by 
\begin{equation*}
    D_\phi(\mu,\nu) 
    \coloneqq
    \int_X \phi \bigl(\tfrac{\dx \mu}{\dx \nu} \bigr) \dx \nu
    + \phi^\prime_\infty \int_X \dx \mu^\perp
\end{equation*}
with the convention $0 \cdot \infty = 0$
and the \textit{recession constant}  
$\phi^\prime_\infty \coloneqq \lim_{x \to \infty} \phi(x)/x$.
  Csiszár divergences are
  jointly convex,  weakly lower semi-continuous,
  and non-negative, see \cite[Cor~2.9]{LMS18}.
In this paper, we will apply the $\phi$-divergences given in \cref{tab:entro}.
\begin{table}[t]
  \centering
  \caption{Entropy functions $\phi$ with recession constant $\phi_\infty'$ and corresponding $\phi$-divergences.}
  \label{tab:entro}
  \begin{tabular}{lccc}
    \toprule
    & $\phi(t)$
    & $\phi_\infty'$
    & $D_\phi(\mu,\nu)$
    \\
    \midrule
    balanced ($\phi_{\bal}$)
    & $\iota_{\{1\}}(t)$
    & $\infty$
    & $\iota_{\{\nu\}}(\mu)$
    \\
    unconstrained ($\phi_{\free}$)
    & $\iota_{[0,\infty)}(t)$
    & $0$
    & $0(\mu,\nu) \equiv 0 $
    \\
    Kullback--Leibler ($\phi_{\ent}$)
    & $t \log(t) - t + 1$
    & $\infty$
    & $\KL(\mu,\nu)$
    \\
    \bottomrule
  \end{tabular}
\end{table}
In particular, for the Boltzman--Shannon entropy $\phi_{\ent}(t) \coloneqq t \log(t) - t + 1$, we obtain
the \emph{Kullback--Leibler (KL) divergence}
\begin{equation*}
  \KL(\mu,\nu)
  \coloneqq
  \begin{cases}
    \int \log \bigl(\frac{\dx \mu}{\dx \nu}\bigr) \dx \nu
    - \mu(X) + \nu(X),
    & \mu \ll \nu,
    \\
    \infty,
    &\text{else}.
  \end{cases}
\end{equation*}

The \emph{regularized, unbalanced GW transport}
\cite{SejViaPey21}
is given for $\varepsilon \ge  0$ by
\begin{equation*}
  \UGW_\eps(\XX_1,\XX_2)
  \coloneqq 
  \!\!\!\!\!
  \inf_{\pi \in \M^+(X_1\times X_2)}
  \Big(\int_{(X_1 \times X_2)^2}
  \!\!\!\!\!
  |d_1 - d_2|^2
  \dx \pi \dx \pi 
	+ \sum_{i=1}^2    D_{\phi_i}^\otimes(\pi_i, \mu_i)
	+ \eps \KL^\otimes(\pi,\mu_1 \otimes \mu_2)
       \Big)^{\frac12}
\end{equation*}
with $D_{\phi_i}^\otimes(\mu,\nu) \coloneqq D_{\phi_i}(\mu \otimes \mu,\nu \otimes \nu)$, $\KL^\otimes \coloneqq D_{\phi_{\text{BS}}}^\otimes$, and marginals 
$\pi_i \coloneqq (P_i)_\# \pi $,
where $P_i(x_1,x_2) \coloneqq x_i$, $i=1,2$.
We use similar notations for projections onto the components of higher-order Cartesian products.
Depending on the chosen divergence,
the marginals of the computed plan may
match the given $\mu_i$ (balanced),
be close to them (scaled Kullback--Leibler),
or completely free (unconstrained). 
In the unbalanced setting, 
the given $\mu_i$ do not have to be probability measures,
i.e.\ $\mu_i \in \M^+(X_i)$.
To keep the notation simple, we use $\mu_1 \otimes \mu_2$ in the KL regularization.
However, this could be replaced by the the product of Lebesgue measures
or counting measures on the respective spaces, see \cite{BLNS2021,NS20}.

%-------------------------------------------------------------------------------
\section{Multi-marginal Gromov--Wasserstein Transport}\label{sec:UMGW}
%-------------------------------------------------------------------------------
Towards a multi-marginal GW setting,
we consider a series of mm-spaces
$\XX_i \coloneqq (X_i,d_i,\mu_i)$, $i=1,\dots, N$.
The Cartesian product of the domains
is henceforth denoted by 
\begin{equation*}
    \vX \coloneqq \bigtimes_{i=1}^N X_i
    \qquad\text{with}\qquad
    \vx \coloneqq (x_1, \dots, x_N),
    \;
    x_i \in X_i.
\end{equation*}
We introduce the \emph{multi-marginal GW transport} problem as
\begin{equation*}
  \MGW(\XX_1,\dotsc,\XX_N)
  \coloneqq \inf_{\pi \in \Pi(\mu_1,\dots,\mu_N)}
  \int_{\vX^2} c(\vx,\vx') \dx \pi(\vx) \dx \pi(\vx'),
\end{equation*}
where $\Pi(\mu_1,\dots,\mu_N) \coloneqq \{\pi \in \p(X) : (P_i)_\# \pi = \mu_i \}$
and $c \colon \vX \times \vX \to [0,\infty]$ is dependent on $d_1,\dotsc,d_N$.
A typical cost function is 
\begin{equation}   \label{eq:cost}
  c(\vx,\vx')
  \coloneqq
  \sum_{i,j=1}^N
  c_{ij} |d_i(x_i,x'_i) - d_j(x_j, x_j')|^2
  \qquad\text{with}\qquad
  c_{ij} \ge 0.
\end{equation}
In analogy to the unbalanced GW transport,
we consider for $\varepsilon \ge 0$, the
\emph{regularized, unbalanced multi-marginal GW transport} problem
\begin{align*}
  &\UMGW_\eps(\XX_1,\dotsc,\XX_N)
  \coloneqq 
  \inf_{\pi \in \M^+(\vX)}
  F_{\eps}(\pi)
  \\
  &F_{\eps}(\pi)
  \coloneqq
    \int_{\vX^2} c(\vx,\vx') \dx \pi(\vx) \dx \pi(\vx')
    + \sum_{i=1}^N D_{\phi_i}^\otimes(\pi_i, \mu_i)
    + \eps \KL^\otimes(\pi,\mu^\otimes),
\end{align*}
where $\mu^\otimes \coloneqq \mu_1 \otimes \cdots \otimes \mu_N$.
Under mild conditions, the infimum is attained.

\begin{proposition}[Existence of Minimizer]     
\label{prop:existence}
    Let the cost function $c \colon \vX \times \vX \to \R_+$
    be lower semi-continuous.  
    Assume
    either $\varepsilon > 0$
    or $\sum_{i=1}^N (\phi_i)'_\infty > 0$.
    Then $\UMGW_\varepsilon$ is finite
    and the infimum is attained.
\end{proposition}

The statement can be shown by
following the lines in the proof of \cite[Prop~7]{SejViaPey21}.
To make the paper self-contained,
we give the proof in the appendix.

%-----------------------------------------------------------------------------
\section{Bi-convex Relaxation}\label{sec:bi-convex-relax}
% -----------------------------------------------------------------------------
In this section, we consider a bi-convex relaxation of $\UMGW_\eps$.
If the cost function belongs to a certain class of conditionally negative definite functions,
then we will see that the relaxation of the balanced problem becomes tight, that is, the values and minimizers of the original problem and the relaxed problem coincide.
We introduce this class of cost functions, before we deal with the bi-convex relaxation.

%--------------------------------------------------------------------------------
\subsection{Conditionally Negative Semi-definiteness}\label{sec:cond-negat-semi}
%--------------------------------------------------------------------------------
In the following, we consider the results from 
\cite{titouan2020co,MarLip18}
from a multi-marginal point of view.
Recall that a symmetric matrix $A \in \mathbb R^{d \times d}$ is 
\emph{negative semi-definite} if $\alpha^\tT A \alpha \le 0$ for all $\alpha  \coloneqq (\alpha(i))_{i=1}^d \in \mathbb R^{d}$.
This is equivalent to the fact that $A$ has only non-positive eigenvalues.
Furthermore, $A \in \mathbb R^{d \times d}$ is \emph{conditionally negative semi-definite} on a subspace $V \subset \mathbb R^{d}$
if $v^\tT A v \le 0$ for all $v \in V$. 
If the columns of $B \in  \mathbb R^{d \times \dim V}$ form an orthonormal basis of $V$, this is equivalent
to $\alpha^\tT B^\tT A B \alpha \le 0$ for all $\alpha \in \mathbb R^{\dim V}$, 
i.e.\ to the negative semi-definiteness of $B^\tT A B \in  \mathbb R^{\dim V \times \dim V}$.

A continuous, symmetric function $k\colon X \times X \rightarrow \mathbb R$ 
is \emph{negative definite}, if for 
all points
$x^1, \ldots, x^d \in X$ and all $d \in \mathbb N$, 
the matrix
$
\left(k(x^m,x^n) \right)_{m,n=1}^d
$
is negative semi-definite. It is called \emph{conditionally negative definite} (of order 1) 
if this holds true on the subspace
\begin{equation*}
V_{{\bf 1}_d}  \coloneqq \{\alpha \coloneqq (\alpha(j))_{j=1}^d \in \mathbb R^d: \alpha^\tT \mathbf 1_d = 0 \},
\end{equation*}
see \cite{Sch38,Wendland:2004wd}.
Equivalently, the above function $k \colon X \times X \to \R$ is
conditionally negative definite
if  
\begin{equation*}
  \sum_{m, n = 1}^d
  k(x^{m}, x^{n}) \,
  \alpha(x^{m}) \alpha(x^{n}) \le 0
\end{equation*}
for all  $\mathcal X \coloneqq \{x^m \in X: m = 1,\ldots,d \}$,  $d \in \N$,
and for all $\alpha \in \mathcal M(\mathcal X)$
satisfying $\alpha(\mathcal X) = 0$.
Typical conditionally negative definite functions are the Euclidean metric $\|x-y\|_2$ on $\mathbb R^d$
and spherical distances $d(x,y)$, see \cite{BBS2007}.
For further examples, we refer to \cite{FerLauHau15}.
We will need the following auxiliary lemma.

\begin{lemma}  \label{lem:tens-cnd}
  Let $A_i \in \R^{d_i\times d_i}$ be symmetric, and let $V_i \subset \R^{d_i}$, $i=1,2$, be subspaces.
  Then $-(A_1 \otimes A_2)$ is conditionally negative semi-definite
  on $V_1 \otimes V_2$ 
  if and only if
  either
  $A_i$, $i=1,2$, or $-A_i$, $i=1,2$, are 
  conditionally negative semi-definite
  on $V_i$, $i=1,2$.
\end{lemma}

\begin{proof}
  Let $B_i \in \R^{d_i \times \dim V_i}$ be matrices
  whose columns form an orthonormal bases of $V_i$, $i=1,2$.
  Then $B_1 \otimes B_2$ is a basis of $V_1 \otimes V_2$
  and   $-(A_1 \otimes A_2)$ is conditionally negative semi-definite on $V_1 \otimes V_2$
  if and only if
  \begin{equation*}
    -(B_1 \otimes B_2)^\tT
    (A_1 \otimes A_2)
    (B_1 \otimes B_2)
    =
    -(B_1^\tT A_1 B_1)
    \otimes
    (B_2^\tT A_2 B_2)
  \end{equation*}
  is negative semi-definite.
  Considering that
  the eigenvalues on the right-hand side are the negative pairwise products
  of the eigenvalues of $B_i^\tT A_i B_i$,  $i=1,2$,
  we obtain the assertion.
\end{proof}

In our application, 
the cost function is given on the compact Cartesian space $\vX \coloneqq \bigtimes_{i=1}^N X_i$, $N \ge 2$
and often has a special structure as the one in \cref{eq:cost}.
This gives rise to the following definition. 
For $M \coloneqq (M_i)_{i=1}^N \in \N^N$,
let $[M] \coloneqq \{ m = (m_1, \ldots,m_N) \in \N^N : 1 \le m_i \le M_i \}$.
A continuous, symmetric function $c \colon \vX \times \vX \to \R$ is called
\emph{marginal conditionally negative definite}
if  
\begin{equation} \label{mcd}
  \sum_{m, n \in [ M]}
  c(\vx^{m}, \vx^{n}) \,
  \alpha(\vx^{m}) \alpha(\vx^{n}) \le 0,
  \qquad \vx^m \coloneqq \left(x_1^{m_1}, \ldots, x_N^{m_N} \right) \in \bigtimes_{i=1}^N \mathcal X_i,
\end{equation}
for all  $\mathcal X_i \coloneqq \{x_i^{m_i} \in X_i: m_i = 1,\ldots,M_i \}$,  $M \in \N^N$
and for all $\alpha \in \mathcal M(\bigtimes{\!}_{i=1}^N \mathcal X_i)$
satisfying $(P_i)_\# \alpha \equiv 0$ for $1 \le i \le N$.
Clearly, a conditionally negative definite function is
marginal conditionally negative definite, but not conversely.
Then we have the following proposition.

\begin{proposition}
  Let $c \colon \vX \times \vX \to \R$ be defined by \cref{eq:cost},
  where $d_i$, $i=1,\dots,N$, are conditionally negative semi-definite functions.
  Then $c$ is marginal conditionally negative definite.
\end{proposition}

\begin{proof}
  Since the sum of marginal conditionally negative definite functions keeps this property,
  it remains to restrict our attention to the function
  $$
  |d_i(x_i,x_i') - d_j(x_j,x_j')|^2 = d_i^2(x_i,x_i') + d_j^2(x_j,x_j') - 2 d_i(x_i,x_i') d_j(x_j,x_j').
  $$
  By definition of $\alpha$, we have that
  $$
  \sum_{m, n \in [M]}
  d_i^2(x_i^{m_i}, x_i^{n_i}) \,
  \alpha(\vx^{m}) \alpha(\vx^{n}) 
  = \sum_{m_i, n_i} d_i^2(x_i^{m_i}, x_i^{n_i})
  \, \bigl[(P_i)_\# \alpha \bigr] (x_i^{m_i})
  \, \bigl[(P_i)_\# \alpha \bigr] (x_i^{n_i})
  = 0
  $$
  and that $-d_i(x_i,x_i') d_j(x_j,x_j')$ is negative semi-definite on $\vX \times \vX$
  if and only if this is true on $(X_i \times X_j) \times (X_i \times X_j)$.
  For the latter, it remains to show that 
  $-( d_i(x_i^{m_i}, x_i^{n_i}))_{m_i,n_i=1}^{M_i} \otimes (d_j(x_j^{m_j}, x_j^{n_j}) )_{m_j,n_j=1}^{M_j}$
  is negative semi-definite on 
  $V_{{\bf 1}_{M_i}} \otimes V_{{\bf 1}_{M_j}}$. By \cref{lem:tens-cnd} this is indeed the case.
\end{proof}

%-----------------------------------------------------------------------------
\subsection{Tight Bi-convex Relaxation}\label{subsec:bi-convex-relax}
% -----------------------------------------------------------------------------
The objective $F_\eps$ in $\UMGW_\eps$ is quadratic in
$\pi$.
Inspired from \cite{SejViaPey21} for GW distances, we propose
the following bi-convex relaxation to tackle $\UMGW_\eps$ numerically:
\begin{align*}  
  &\UMGW_\eps^{\bi}(\XX_1,\dotsc,\XX_N)
  \coloneqq \inf_{\pi,\gamma \in \M^+(\vX)}
  F^{\bi}_\eps(\pi,\gamma)
  \\
  &F^{\bi}_\eps(\pi,\gamma)
  \coloneqq
    \int_{\vX^2} c(\vx,\vx') \dx \pi(\vx) \dx \gamma(\vx')
    + \sum_{i=1}^N D_{\phi_i}(\pi_i \otimes \gamma_i, \mu_i \otimes \mu_i)
    + \eps \KL\bigl(\pi \otimes \gamma,(\mu^\otimes)^2\bigr).
\end{align*}
By construction,
the relaxation always satisfies $\UMGW_\eps^{\bi} \le \UMGW_\eps$.
Moreover, in the balanced case, it
becomes tight
for marginal conditionally negative definite cost functions,
meaning that any minimizer $(\pi^*,\gamma^*)$ of $\UMGW_\eps^{\bi}$
yields the minimizers $\pi^*$ and $\gamma^*$ of $\UMGW_\eps$.

\begin{proposition}[Tightness of Relaxation]  \label{thm:tight}
  Let $\phi_i \coloneqq \phi_{ \bal }$, $i=1,\ldots,N$, $\eps \ge 0$
  and  $c$ be marginal conditionally negative definite.
  Then, $\UMGW_\eps=\UMGW_\eps^{\bi}$
  and   every minimizer $(\pi^*, \gamma^*)$ of $\UMGW_\eps^{\bi}$ satisfies
  \begin{equation}
    \label{eq:equiv-min}
    F_\eps^{\bi}(\pi^*, \gamma^*)
    = F_\eps^{\bi}(\pi^*, \pi^*)
    = F_\eps^{\bi}(\gamma^*, \gamma^*)
  \end{equation}
  such that $\pi^*$ and $\gamma^*$ are minimizers of $\UMGW_\eps$.
\end{proposition}

\begin{proof}
  The  balanced setting ensures that
  all occurring measures are probability measures.
  The KL divergence then splits into
  $\KL(\pi \otimes \gamma, \mu^\otimes \otimes \mu^\otimes)
  = \KL(\pi, \mu^\otimes) + \KL(\gamma, \mu^\otimes)$
  by \cite[Prop~9]{SejViaPey21}.
  Furthermore,
  the values in \cref{eq:equiv-min} always satisfy
  \begin{equation*}
    F_\eps^{\bi}(\pi^*, \gamma^*)
    \le
    F_\eps^{\bi}(\pi^*, \pi^*)
    \qquad\text{and}\qquad
    F_\eps^{\bi}(\pi^*, \gamma^*)
    \le
    F_\eps^{\bi}(\gamma^*, \gamma^*)
  \end{equation*}
  Incorporating the KL splitting,
  and defining
  $K(\pi, \gamma) \coloneqq \int_{\vX^2} c \dx \pi \dx \gamma$,
  we obtain
  \begin{align}
    K(\pi^*, \gamma^*) + \eps \KL(\pi^*, \mu^\otimes) + \eps \KL(\gamma^*, \mu^\otimes)
    &\le
    K(\pi^*, \pi^*) + 2\eps \KL(\pi^*, \mu^\otimes), \label{eq:tight_ineq1}
    \\
    K(\pi^*, \gamma^*) + \eps \KL(\pi^*, \mu^\otimes) + \eps \KL(\gamma^*, \mu^\otimes)
    &\le
    K(\gamma^*, \gamma^*) + 2\eps \KL(\gamma^*, \mu^\otimes). \label{eq:tight_ineq2}
  \end{align}
  Now, assume that \cref{eq:equiv-min} is false,
  i.e.\ at least one of the inequalities \cref{eq:tight_ineq1,eq:tight_ineq2} is strict.
  Summing both inequalities then yields
  \begin{equation}
    \label{eq:contra}
    0 < K(\pi^* - \gamma^*, \pi^* - \gamma^*)
    =
    K(\alpha,\alpha)
    \qquad\text{with}\qquad
    \alpha \coloneqq \pi^* - \gamma^*.
  \end{equation}
  Since $\pi^*, \gamma^* \in \Pi(\mu_1, \dots, \mu_N)$,
  we have 
  $(P_i)_\# \alpha \equiv 0$, $i=1,\ldots,N$.
  Then the marginal conditionally negative definiteness of $c$
  ensures $K(\alpha, \alpha) \le 0$.
  This is however a contradiction to \cref{eq:contra}.
\end{proof}

Note that the proposition does not immediately follow from \cite[Thm~2]{SejViaPey21}
since the proof given there is not reproducible in the non-definite case. More precisely, 
$x^\tT A x = 0$, $x^\tT \bf 1 = 0$ does not imply $x \in \text{ker} A$ for a conditionally negative semi-definite matrix $A$.

The bi-convex relaxation $\UMGW_\eps$ encourages to
use an alternating minimization scheme over $\pi$ and $\gamma$.
Here the minimization over one variable corresponds to
a multi-marginal optimal transport problem
with a specific cost function.
In the following proposition, 
we restrict our attention to the divergences from \cref{tab:entro}.

\begin{proposition}   \label{prop:relax_as_ot}
  Let $\phi_i$, $i =1,\dotsc,N$, be an entropy function from \cref{tab:entro},
  and set $\mathcal I \coloneqq \{ i : \phi_i = \phi_{\ent} \}$.
  Then, 
  for fixed $\gamma \in \M^+(X)$ 
  with $D_{\phi_i}(\gamma_i, \mu_i) < \infty$
  and $\KL(\gamma, \mu) < \infty$,
  the minimization over $\pi$ in $\UMGW_\eps^{\bi}$
  becomes the multi-marginal transport problem
  \begin{equation*}
    \argmin_{\pi \in \M^+(\vX)} F^{\bi}_{\eps}(\pi,\gamma)
    \coloneqq
     \int_X c_{\gamma}(\vx) \dx \pi(\vx)
    + \|\gamma\|_{\TV} \, \Bigl( \sum_{i=1}^N D_{\phi_i}(\pi_i,\mu_i) \Bigr)
    + \eps \|\gamma\|_{ \TV } \KL(\pi, \mu^\otimes)
  \end{equation*}
  with
  \begin{equation}
    \label{eq:ot-cost}
    c_{\gamma}(\vx) 
    \coloneqq
    \int_{\vX} c(\vx,\vx') \dx \gamma(\vx')
    + \Bigl[ \sum_{i \in \mathcal I}
    \int_{X_i}
    \log \bigl(\tfrac{\dx \gamma_i}{\dx \mu_i}\bigr) \dx \gamma_i
    \Bigr]
    + \eps
    \int_{\vX}
    \log \bigl(\tfrac{\dx \gamma}{\dx \mu^\otimes}\bigr) \dx \gamma. 
  \end{equation}
\end{proposition}

\begin{proof}
  For arbitrary measures $\alpha_1$, $\alpha_2$, $\alpha_3 \in \M^+(Y)$,
  the KL divergence factorizes into 
  \begin{align*}
    \KL(\alpha_1 \otimes \alpha_2, \alpha_3 \otimes \alpha_3)
    &=
    \| \alpha_1 \|_{\TV} \, \KL(\alpha_2, \alpha_3)
    + \|\alpha_2\|_{\TV} \, \KL(\alpha_1, \alpha_3)
    \\
    &\qquad + (\| \alpha_1 \|_{\TV} - \| \alpha_3 \|_{\TV}) \,
    (\| \alpha_2 \|_{\TV} - \| \alpha_3 \|_{\TV}),
  \end{align*}
  see \cite[Prop~9]{SejViaPey21}.
  Using the definition of Csiszár divergences, 
  we especially have
  \begin{equation}
  \label{eq:fac-kl}
  \begin{aligned}
    \KL(\alpha_1 \otimes \alpha_2, \alpha_3 \otimes \alpha_3)
    &=
    \| \alpha_1 \|_{\TV} \, \int_Y \log (\tfrac{\dx \alpha_2}{\dx \alpha_3}) \dx \alpha_2
    + \|\alpha_2\|_{\TV} \, \KL(\alpha_1, \alpha_3)
    \\
    &\qquad- \| \alpha_3 \|_{\TV} \,
    (\| \alpha_2 \|_{\TV} - \| \alpha_3 \|_{\TV}).
  \end{aligned}
  \end{equation}
  Assume $D_\phi(\alpha_2, \alpha_3) < \infty$,
  we can factorize the balanced and unconstrained divergences into
  \begin{align}
    \label{eq:fac-bal}
      \iota_{\{\alpha_3 \otimes \alpha_3\}}(\alpha_1 \otimes \alpha_2)
      &= \| \alpha_1 \|_{\TV} \, \iota_{\{\alpha_3\}} (\alpha_2)
      + \| \alpha_2 \|_{\TV} \, \iota_{\{\alpha_3\}} (\alpha_1)
      =  \| \alpha_2 \|_{\TV} \, \iota_{\{\alpha_3\}} (\alpha_1)
      \\
      \label{eq:fac-free}
      0(\alpha_1 \otimes \alpha_2, \alpha_3 \otimes \alpha_3)
      &= \| \alpha_1 \|_{\TV} \, 0(\alpha_2, \alpha_3)
      + \| \alpha_2 \|_{\TV} \, 0(\alpha_1, \alpha_3)
      = \| \alpha_2 \|_{\TV} \, 0(\alpha_1, \alpha_3).
  \end{align}
  Applying \cref{eq:fac-kl,eq:fac-bal,eq:fac-free},
  and exploiting  $D_{\phi_i}(\gamma_i, \mu_i) < \infty$
  and $\KL(\gamma, \mu) < \infty$,
  we have
  \begin{align*}
    F_\eps^{\bi}(\pi,\gamma)
    &=\int_{\vX^2} c(\vx,\vx')  \dx \pi(\vx) \dx \gamma(\vx') 
    + \sum_{i \not\in \mathcal I} \|\gamma_i\|_{\TV} \, D_{\phi_i}(\pi_i, \mu_i)
    \\
    &\;+
      \sum_{i\in\mathcal I} \Bigl(
      \int_{\vX} 
      \int_{X_i} \log \bigl(\tfrac{\dx \gamma_i}{\dx \mu_i} \bigr) \dx \gamma_i
      \dx \pi
      + \|\gamma_i\|_{\TV} \, \KL(\pi_i,\mu_i)
      + \underbrace{
      \|\mu_i\|_{\TV}^2 - \|\gamma_i\|_{\TV} \, \|\mu_i \|_{\TV}
      }_{\text{constant}}
      \Bigr)
    \\
    &\;+
      \eps \, \Bigl(
      \int_{\vX} 
      \int_{\vX} \log \bigl(\tfrac{\dx \gamma}{\dx \mu^\otimes} \bigr)
      \dx \gamma
       \dx \pi
      + \|\gamma\|_{\TV} \, \KL(\pi,\mu^\otimes)
      + \underbrace{
      \|\mu^\otimes\|_{\TV}^2 - \|\gamma \|_{\TV} \, \|\mu^\otimes\|_{\TV}
      }_{\text{constant}} \Bigr).
  \end{align*}
  Rearranging terms and omitting the constant terms yields the assertion.
\end{proof}

The minimization of $F^{\bi}_{\eps}(\cdot,\gamma)$ over $\M^+(\vX)$ may be efficiently solved
using the multi-marginal Sinkhorn scheme in \cite{BLNS2021}
if the cost function has a sparse structure.
For example,
we can use cost functions of the form \cref{eq:cost}
with $N-1$ non-zero coefficients $c_{ij}$
such that \cref{eq:cost} decouples according to a tree, that is, there is no sequence $i_1,\dotsc,i_k$ with $c_{i_1,i_2},\dotsc,c_{i_{k-1},i_k} > 0$ and $i_1 = i_k$.
Similarly to \cite[Alg~1]{SejViaPey21}, 
the bi-convex relaxation leads to the alternating minimization scheme in \cref{alg:UMGW}. 
Since $F_\eps^{\bi}$ is invariant
under $(\pi, \gamma) \mapsto (t \, \pi, \frac1t \, \gamma)$ for $t > 0$,
the measures may be balanced such that $\|\pi\|_{\TV} = \|\gamma\|_{\TV}$.
In the algorithm, 
the rescaling of the newly computed variable corresponds this balancing. 
The algorithm produces a sequence $\pi_1,\gamma_1,\pi_2,\gamma_2,\dots$ which monotonously decreases the lower-bounded objective $F^{\bi}_{\eps}$. Hence, the sequence $(F^{\bi}_{\eps}(\pi_k,\gamma_k))_{k \in \N}$ converges. However, in general it may happen that $(\pi_k,\gamma_k)_{k \in \N}$ do not converge.  Even if the iterates converge,
i.e.\ $\pi_k \to \pi$, $\gamma_k \to \gamma$, the limits may not coincide. However, in our numerical experiments,
we usually observe convergence to a single limit $\pi = \gamma$.

\begin{algorithm}
	\begin{algorithmic}
	    \State \textbf{Input:} 
	    \parbox[t]{300pt}{mm-spaces $\XX_i \coloneqq (X_i,d_i,\mu_i)$, $i=1,\dots,N$;\\
        cost function $c$;
	    \\
	    entropy functions $\varphi_i$, $i=1,\dots,N$, from \cref{tab:entro}; 
	    \\
	    regularization parameter $\eps > 0$.}
	    \State \textbf{Initialize} 
	    $\pi \coloneqq \gamma \coloneqq \bigotimes_{i=1}^N\mu_i$.
	    \While{not converged}
	        \State Compute $c_\gamma$ in \cref{eq:ot-cost}.
	        \State Update $\pi$ using the Sinkhorn scheme in \cite{BLNS2021} with $c_\gamma$. 
	        \State Rescale by $\pi \gets \sqrt{\|\gamma\|_{\TV} / \|\pi\|_{\TV}} \, \pi$.
	        \State Compute $c_\pi$ in \cref{eq:ot-cost}.
	        \State Update $\gamma$ using the Sinkhorn scheme in \cite{BLNS2021} with $c_\pi$. 
	        \State Rescale by $\gamma \gets \sqrt{\|\pi\|_{\TV} / \|\gamma\|_{\TV}} \, \gamma$.
	    \EndWhile
     	\State \textbf{Output:} $(\gamma,\pi)$
	\end{algorithmic}
	\caption{(Unbalanced) Multi-marginal Gromov--Wasserstein Transport}
	\label{alg:UMGW}
\end{algorithm}

%-------------------------------------------------------------------------------
\section{Barycenters} \label{sec:bary}
%-------------------------------------------------------------------------------
Barycenters and multi-marginal optimal transport problems are closely related \cite{ABM16brute,BLNS2021,CE10,tree21HRCK}.
For the GW setting,
we obtain similar results.
For $\rho_i \ge 0$ with $\sum_{i=1}^N \rho_i = 1$,
a \emph{Gromov--Wasserstein barycenter}
between the mm-spaces $\XX_i \coloneqq (X_i, d_i, \mu_i)$,
$i = 1, \dots, N$,
is a minimizing mm-space $\YY \coloneqq (Y,d,\nu)$ of
\begin{equation}
  \label{eq:gw_bary}
  \inf_{\YY} \sum_{i=1}^N \rho_i \GW^2(\XX_i, \YY).
\end{equation}
Indeed the next result shows that solutions always exist and characterizes them via multi-marginal solutions.

\begin{theorem}[Free-Support Barycenter]   \label{thm:gen-bary}
  Let $\XX_i$ be given mm-spaces,
  and let $\rho_i \ge 0$ be weights with $\sum_{i=1}^N \rho_i = 1$. Then the infimum in \cref{eq:gw_bary} is attained. 
  Moreover,
  each solution of \cref{eq:gw_bary} is isomorphic to an mm-space of the form $\YY^* \coloneqq (\vX, d^*, \pi^*)$, where
  $d^*(\vx,\vx') \coloneqq \sum_{i=1}^N \rho_i d_i(x_i,x_i')$,
  and $\pi^*$ is a minimizer of $\MGW(\XX_1, \dots, \XX_N)$
  with cost function
  \begin{equation*}
    c(\vx,\vx')
    \coloneqq
    \frac12
    \sum_{i,j=1}^N
    \rho_i \rho_j \,
    \lvert d_i(x_i,x_i') - d_j(x_j,x_j') \rvert^2.
  \end{equation*}
\end{theorem}

\begin{proof}
  First we note that
  the cost function may be rearranged as
  \vspace{-15pt}
  \begin{equation*}
    c(\vx,\vx')
    =
    \sum_{i=1}^N \rho_i \, d_i^2(x_i,x_i')
    - \sum_{i,j=1}^N \rho_i \rho_j \, d_i(x_i,x_i') \, d_j(x_j,x_j')
    =
    \sum_{i=1}^N
    \rho_i \, \Bigl| 
    d_i(x_i,x_i') 
    - \overbrace{
      \sum_{j=1}^N \rho_j \, d_j(x_j,x_j')
    }^{= d^*(x,x')}
    \Bigr|^2.
  \end{equation*}
  Let $\YY = (Y,d,\nu)$ be an arbitrary mm-space,
  and let $\pi^{(i)}_{\GW}$ be an optimal plan of $\GW(\XX_i,\YY)$, $i = 1,\dotsc,N$.
  Since all these plans have the marginal
  \raisebox{0pt}[0pt][0pt]{$(P_2)_\# \pi^{(i)}_{\GW} = \nu$},
  Dudley's lemma \cite[Lem~8.4]{ABS21} ensures
  the existence of a gluing
  $\pi_{\mathrm{g}} \in \Pi(\mu_1,\dotsc,\mu_N,\nu)$
  with \raisebox{0pt}[0pt][0pt]{$(P_{X_i \times Y})_\# \pi_{\mathrm{g}} = \pi^{(i)}_{\GW}$}.
  Exploiting that
  the 2-marginals of $\pi_{\mathrm{g}}$ are optimal Gromov--Wasserstein plans,
  and that
  the mean $d^*(\vx,\vx')$ is the pointwise minimizer of
  \raisebox{0pt}[0pt][0pt]{$\min_{t \in \R} \sum_{i=1}^N \rho_i \lvert d_i(x_i,x_i') - t \rvert^2$}
  for fixed $\vx$ and $\vx'$, we obtain
  \begin{equation}
    \begin{aligned}
      &\sum_{i=1}^N \rho_i \GW^2(\XX_i,\YY)
      =
      \int_{(\vX \times Y)^2}
      \sum_{i=1}^N \rho_i \,
      \lvert d_i(x_i,x_i') - d(y,y') \rvert^2
      \dx \pi_{\mathrm{g}}(\vx,y)\dx \pi_{\mathrm{g}}(\vx',y')
      \\
      &\geq
      \int_{(\vX \times Y)^2}
      \underbrace{
        \sum_{i=1}^N \rho_i
        \lvert d_i(x_i,x_i') - d^*(\vx,\vx') \rvert^2
      }_{=c(\vx,\vx')}
      \dx \pi_{\mathrm{g}}(\vx,y)\dx \pi_{\mathrm{g}}(\vx',y')
      \ge \MGW(\XX_1, \dots, \XX_N),
    \end{aligned}
    \label{eq:gw-mgw}
  \end{equation}
  where we use $(P_{\vX})_\#\pi_{\mathrm{g}} \in \Pi(\mu_1, \dots, \mu_N)$ in the last step.
  Using the substitutions $\pi^{(i)}_* \coloneqq ({P_{X_i}},\id)_\# \pi^*$
  for any optimal plan $\pi^*$ of $\MGW(\XX_1,\dotsc,\XX_N)$,
  we further have
  \begin{equation}
  \label{eq:mgw-gw}
  \begin{aligned}
    &\MGW(\XX_1,\dotsc,\XX_N)
      =
      \sum_{i=1}^N \rho_i
      \int_{\vX^2}
      \lvert d_i(x_i,x_i') - d^*(x,x') \rvert^2
      \dx \pi^*(\vx) \dx \pi^*(\vx')\\
    &=
      \sum_{i=1}^N \rho_i
      \int_{(X_i \times \vX)^2}
      \lvert d_i(x_i,x_i') - d^*(\vy,\vy') \rvert^2
      \dx \pi^{(i)}_*(x_i,\vy)\dx \pi^{(i)}_*(x_i',\vy')
      \geq
      \sum_{i=1}^N \rho_i \GW^2(\XX_i, \YY^*).
  \end{aligned}
  \end{equation}
  Due to \cref{eq:gw-mgw},
  the last inequality has to be an equality
  showing that
  $\YY^* = (\vX, d^*, \pi^*)$
  from the assertion
  minimizes \cref{eq:gw_bary} and is
  in fact a barycenter between $\XX_1, \dots, \XX_N$.

  The other way round,
  \cref{eq:mgw-gw} implies that
  \cref{eq:gw-mgw} becomes an equality 
  for every further minimizer $\YY^\dagger = (Y^\dagger, d^\dagger, \nu^\dagger)$ of \cref{eq:gw_bary}.
  Analogously to above,
  for every constructed gluing 
  \smash{$\pi^\dagger_{\mathrm g} \in \Pi(\mu_1, \dots, \mu_N, \nu^\dagger)$},
  the plan
  $(P_{\vX})_\# \pi_{\mathrm{g}}^\dagger$
  is a solution of $\MGW(\XX_1,\dotsc,\XX_N)$.
  We will show that
  $\YY^\dagger$ is isomorphic to $\YY^* \coloneqq (\vX, d^*, \pi^*)$
  with \smash{$\pi^* \coloneqq (P_{\vX})_\# \pi_{\mathrm{g}}^\dagger$}.
  Exploiting the equality in \cref{eq:gw-mgw}, that
  the mean $d^*(\vx,\vx')$ is the pointwise minimizer of
  \smash{$\min_{t \in \R} \sum_{i=1}^N \rho_i \lvert d_i(x_i,x_i') - t \rvert^2$},
  and that the integrands are non-negative,
  we firstly conclude
  \begin{equation*}
      \sum_{i=1}^N \rho_i \,
      \lvert d_i(x_i,x_i') - d^\dagger(y,y') \rvert^2
      =
      \sum_{i=1}^N \rho_i
      \lvert d_i(x_i,x_i') - d^*(\vx,\vx') \rvert^2
      \quad
      \text{for }(\pi_{\mathrm{g}}^\dagger \otimes \pi_{\mathrm{g}}^\dagger)\text{-a.e.\ } (\vx,y,\vx',y')
  \end{equation*}
  and secondly $d^\dagger(y,y') = d^*(\vx,\vx')$ for $\pi_{\mathrm{g}}^\dagger \otimes \pi_{\mathrm{g}}^\dagger$-a.e.\ $(\vx,y,\vx',y')$.
  Since $\pi_{\mathrm{g}} \in \Pi(\pi^*,\nu^\dagger)$, we obtain
  \begin{equation*}
  \GW^2(\YY^*,\YY^\dagger) \leq \int_{\vX \times Y^\dagger} 
  \lvert 
  \underbrace{d^*(\vx,\vx') - d^\dagger(y,y')}_{\equiv 0 \; \text{a.e.}}
  \rvert^2
  \dx \pi_{\mathrm{g}}(\vx,y) \pi_{\mathrm{g}}(\vx',y') = 0,
  \end{equation*}
  which finally shows that $\YY^*$ and $\YY^\dagger$ are isomorphic.
\end{proof}

In the special case $N=2$,
the barycenters
from \cref{thm:gen-bary}
have the form
$(X_1 \times X_2,(1-\rho) \, d_1 + \rho \, d_2, \pi^*)$,
where $\pi^*$ is an optimal GW plan.
Each optimal plan $\pi^*$ thus corresponds to a geodesic in the GW space
\cite[Thm~3.1]{sturm2020space}.
Although \cref{thm:gen-bary} allows us
to determine barycenters between arbitrary spaces,
due to the generality of the mm-space $\YY$,
these barycenters are difficult to interpret
since $\YY$ can have a completely different structure than $\XX_i$.
However, for GW barycenters with respect to images for instance,
it is desirable to obtain again an image.
In this situation,
it therefore makes sense
to fix $(Y,d)$ in $\YY$
and to minimize only over the measure $\nu$.
Moreover, the GW barycenter may be relaxed
by considering unbalanced transport.
Against this background,
we consider the
\emph{fixed-support (unbalanced) Gromov--Wasserstein barycenter}
given by 
\begin{equation}
  \label{eq:rest-bary}
  \argmin_{\nu \in \M^+(Y)}
  \sum_{i=1}^N \rho_i \UGW^2(\XX_i,\YY)
  \qquad\text{with}\qquad
  \YY = (Y, d, \nu),
\end{equation}
where the $\UGW$ terms may be unbalanced in $\XX_i$
with respect to some entropy function $\phi_i$
and are balanced with respect to $\YY$. 
In the following, we consider unbalanced multi-marginal transports where one input is unconstrained so that the corresponding marginal is completely free. To emphasize the independence of the input measure, we introduce the notation $\bullet$, which acts as a dummy measure on an associated compact metric space $(Y,d)$ so that $(Y,d,\bullet)$ becomes an mm-space.

\begin{theorem}[Fixed-Support Barycenter]
  \label{thm:rest-bary}
  Let $\XX_i$ be given mm-spaces,
  $(Y,d)$ be a given metric space,
  and $\rho_i \ge 0$ be weights
  with $\sum_{i=1}^N \rho_i = 1$. Then the infimum in \cref{eq:rest-bary} is attained by
  $\YY^* \coloneqq (Y, d, \nu^*)$,
  with $\nu^* \coloneqq (P_Y)_\# \pi^*$,
  where $\pi^*$ minimizes
  $\UMGW(\XX_1,\dots,\XX_N,(Y,d,\bullet))$
  with cost function
  \begin{equation}\label{eq:star_cost}
    c((\vx,y),(\vx',y'))
    \coloneqq
    \sum_{i=1}^N \rho_i \,
    | d_i(x_i,x_i') - d(y,y')|^2.
  \end{equation}
  The unregularized UMGW ($\eps = 0$) is  unbalanced in $\XX_i$
  with respect to $\rho_i\phi_i$
  and unrestricted in $(Y,d,\bullet)$. Furthermore, for any arbitrary minimizer $\YY^\dagger = (Y,d,\nu^\dagger)$, there exists a multi-marginal solution $\pi^*$ of $\UMGW(\XX_1,\dotsc,\XX_N,(Y,d,\bullet))$ so that $(P_{Y})_\# \pi^* = \nu^\dagger$.
\end{theorem}

\begin{proof}
  The assertion may be established similarly to \cref{thm:gen-bary}.
  First, let $\nu$ be an arbitrary measure on $(Y,d)$, and set $\YY \coloneqq (Y,d,\nu)$.
  Let \smash{$\pi^{(i)}_{\UGW}$} be an optimal plan of $\UGW(\XX_i,\YY)$, $i=1,\ldots,N$.
  Since the $\UGW$ problems are balanced in $\YY$,
  all these plans have the marginal
  \raisebox{0pt}[0pt][0pt]{$(P_2)_\# \pi^{(i)}_{\UGW} = \nu$}.
  Again,
  Dudley's lemma \cite[Lem~8.4]{ABS21} ensures
  the existence of a gluing
  $\pi_{\mathrm{g}} \in \M^+(\vX \times Y)$
  with \raisebox{0pt}[0pt][0pt]{$(P_{X_i \times Y})_\# \pi_{\mathrm{g}} = \pi^{(i)}_{\UGW}$}.
  Based on this gluing,
  we obtain
  \begin{equation}
    \label{eq:ugwb-umgw}
    \begin{aligned}
      &\sum_{i=1}^N \rho_i \UGW^2(\XX_i, \YY)
      =
      \sum_{i=1}^N \rho_i \,
      \Bigl(
      \int_{(\vX \times Y)^2}
      |d_i - d|^2
      \dx \pi_{\mathrm{g}} \dx \pi_{\mathrm{g}}
      + D_{\phi_i}^\otimes((P_{X_i})_\# \pi_{\mathrm{g}}, \mu_i)
      \Bigr)
      \\
      &=
      \int_{(\vX \times Y)^2}
      \!\underbrace{
      \sum_{i=1}^N \rho_i \,
      |d_i - d|^2
      }_{=c}
      \dx \pi_{\mathrm{g}} \dx \pi_{\mathrm{g}}\!
      + \!\sum_{i=1}^N \rho_i \,
      D_{\phi_i}^\otimes((P_{X_i})_\# \pi_{\mathrm{g}},\mu_i)
      \ge
      \UMGW(\XX_1,\dots,\XX_N,(Y,d,\bullet)),
    \end{aligned}
  \end{equation}
  where $\UMGW$ is unrestricted in $(Y,d,\bullet)$.
  Considering the marginals
  $\pi^{(i)}_* \coloneqq (P_{X_i \times Y})_\# \pi^*$
  of any optimal $\UMGW$ plan $\pi^*$,
  whose second marginals coincide,
  we further have
  \begin{align*}
    \UMGW(\XX_1,\dots,\XX_N,(Y,d,\bullet))
    &=
      \sum_{i=1}^N \rho_i \,
      \Bigl(
      \int_{(X_i \times Y)^2}
      |d_i - d|^2
      \dx \pi^{(i)}_* \dx \pi^{(i)}_*
      + D_{\phi_i}^\otimes((P_1)_\# \pi^{(i)}_*,\mu_i)
      \Bigr)
    \\
    &\ge 
      \sum_{i=1} \rho_i \UGW^2(\XX_i,\YY^*).
  \end{align*}
  The last inequality has to be an equality because of \cref{eq:ugwb-umgw},
  so $\YY^*$ with $\nu^* = (P_Y)_\# \pi^*$ is a fixed-support barycenter.
  Finally, for some arbitrary barycenter $\YY^\dagger = (Y,d,\nu^\dagger)$, we can construct a gluing $\pi_{\mathrm{g}}^\dagger \in \M^+(\vX \times Y)$ with $(P_{Y})_\# \pi_{\mathrm{g}}^\dagger = \nu^\dagger$ as above. 
  Repeating the previous arguments for $\pi_{\mathrm{g}}^\dagger$, 
  we obtain equality in \cref{eq:ugwb-umgw}; so $\pi_{\mathrm{g}}^\dagger$ is a solution to $\UMGW(\XX_1,\dotsc,\XX_N,(Y,d,\bullet))$, which concludes the proof.
\end{proof}

\begin{remark}[Fixed-Support Barycenter Computation with \cref{alg:UMGW}]\label{rem:umgw_for_bary}
    Combining \cref{alg:UMGW} and \cref{thm:rest-bary} yields a numerically tractable procedure for the (unbalanced) barycenter computation of $N$ inputs $\XX_1,\dotsc,\XX_N$. This is achieved by running the algorithm with the $N+1$ inputs 
    $\XX_1,\dotsc,\XX_N$, and $\YY \coloneqq (Y,d,\bullet)$, where $\YY$ is an mm-space supported on some a priori fixed metric space $(Y,d)$, 
    with $c$ as defined in \cref{eq:star_cost},
    with entropy functions $\phi_1,\dotsc,\phi_N,\phi_{\free}$,
    and with small regularization $\eps > 0$.
    For an output $\pi$, the barycenter is obtained as $(Y,d,(P_Y)_\# \pi)$.
\end{remark}

\begin{remark}[Comparison with the Procedure from \cite{PCS2016}]\label{rem:peyre}
    In contrast to our approach in \cref{rem:umgw_for_bary},
    Peyré, Cuturi, and Solomon \cite{PCS2016} propose to
    compute a (regularized) GW barycenter $\YY \coloneqq(Y,d,\nu)$
    by fixing the number of points in $Y \coloneqq \{y_1, \dots, y_m\}$ 
    and the measure $\nu \in \p(Y)$ beforehand.
    Thus it remains to determine the metric $d$
    or equivalently the \emph{dissimilarity matrix}
    $D_{d} \coloneqq (d(y_i, y_j))_{i,j=1}^m$ by minimizing
    \begin{equation*}
      \argmin_{D_{d} \in \R^{m \times m}}
      \sum_{i=1}^N \rho_i \GW^2_\eps (\XX_i,\YY)
      \quad\text{with}\quad
      \YY = (Y,d,\nu),
    \end{equation*}
    where $\GW_\eps$ is the balanced version of $\UGW_\eps$ with $\phi_\bal$.
    For this purpose,
    Peyré et al. propose a block-coordinate descent, 
    which alternatively minimizes over $D_{d}$ and 
    the (regularized) $\GW_\eps$ plans $(\pi_i)_{i=1}^N$ between $\XX_i$ and $\YY$.
    Intriguingly, the minimization with respect to $D_{d}$ can be given in closed form
    and is numerically negligible.
    The minimization with respect to $(\pi_i)_{i=1}^N$ is achieved by an $N$-fold application of projected gradient descent. 
    For a suitable step-size, a single descent step may be efficiently computed 
    using the Sinkhorn algorithm,
    and a close inspection of the resulting algorithm shows that
    the descent steps essentially coincide with the update of $\pi$ and also $\gamma$ 
    from \cref{alg:UMGW} for the bi-marginal, discrete setting,
    see also \cite{SejViaPey21}. 
    Due to the tree-structured cost \cref{eq:star_cost},
    the complexity per iteration of the multi-marginal Sinkhorn is the same 
    as $N$ iterations of the bi-marginal Sinkhorn \cite{BLNS2021}. 
    Since the constant terms in \cref{eq:ot-cost} can be neglected for the balanced case,
    and due to the manner 
    how the multi-marginal plan is stored 
    as set of bi-marginal plans in \cite{BLNS2021},
    the complexity to compute $c_\gamma$ for $c$ as in \cref{eq:star_cost} for $N$ marginals is the same as an $N$-fold computation of $c_\gamma$ for two marginals.
    For the computation of balanced barycenters,
    the numerical complexity of the approach in \cref{rem:umgw_for_bary}
    and the algorithm in \cite{PCS2016} is thus comparable.
    A numerical run-time comparison is given in \cref{sec:numerics}.
\end{remark}

As described in \cref{rem:peyre},
the output of the procedure in \cite{PCS2016} is a dissimilarity matrix. 
If we want to determine a GW barycenter on a specific space,
it is common practice to embed the points $Y \coloneqq \{y_1, \dots, y_m\}$ 
into the required metric space such that the distances between them (nearly)
coincide with the computed dissimilarity matrix. 
Depending on which space one is interested in, 
this may pose some challenges.
Using our procedure circumvents this embedding-issue
because the fixed-support barycenters are directly constructed 
with respect to the required space.

\section{Multi-marginal Fused Gromov--Wasserstein Transport}\label{sec:fused}

The multi-marginal Gromov--Wasserstein transport allows 
to compare the structure information between several objects.
In many applications,
besides the structure information,
there is also label information available. 
For instance, the pixels of a colour image are labelled with colour information.
In medical applications, 
given data may additionally contain information about the corresponding tissue.
In graph theory, 
the nodes are sometimes annotated or possess additional labels. 
The major motivation behind the \emph{fused Gromov--Wasserstein distance} \cite{vayer2020fused}
is to combine the structure and label information. 
In this section,
we will generalize the fused Gromov--Wasserstein distance to the multi-marginal setting. 

A \emph{labelled mm-space} is a triple $\Xf \coloneqq (X \times A, d, e, \mu)$
consisting of compact metric spaces $(X,d)$ and $(A,e)$ 
and a Borel probability measure $\mu \in \p(X \times A)$.
The metric space $(X,d)$ contains the structure information,
and $(A,e)$ the label information. 
For the labelled mm-spaces 
$\Xf_1 \coloneqq (X_1 \times A, d_1, e, \mu_1)$ and $\Xf_2 \coloneqq (X_2 \times A, d_2, e, \mu_2)$
on $(A,e)$,
the \emph{fused Gromov--Wasserstein distance} is defined by
\begin{align*}
\FGW_\beta(\Xf_1,\Xf_2)
&\coloneqq
\smashoperator[l]{\inf_{\pi \in \Pi(\mu_1,\mu_2)}}
 \Bigl(
 \int_{((X_1 \times A) \times (X_2 \times A))^2} 
 (1-\beta) \,
 |d_1(x_1,x_1') - d_2(x_2, x_2')|^2
 \\
 &\qquad
 + \tfrac\beta2 \,
  e^2(a_1,a_2)
  + \tfrac\beta2 \,
  e^2(a_1',a_2')
  \dx \pi((x_1,a_1),(x_2,a_2))
  \dx \pi((x_1',a_1'),(x_2',a_2'))
  \Bigr)^{\frac12}
\end{align*}
with respect to the \emph{trade-off parameter} $\beta \in [0,1]$.
The fused Gromov--Wasserstein distance thus combines
the Gromov--Wasserstein distance ($\beta = 0$) on the structure spaces 
with the Wasserstein distance ($\beta = 1$) on the label space. 
Notably, it holds $\FGW_{\beta}(\Xf_1,\Xf_2) = 0$ 
if and only if
there exists 
$\mathcal{I}  
= (\mathcal{I}_1,\mathcal{I}_2)
\colon X_1 \times A \to X_2 \times A$, 
$\mathcal{I}(x,a) = (\mathcal{I}_1(x),\mathcal{I}_2(a))$ such that $\mathcal{I}_\# \mu = \nu$, $\mathcal{I}_1:(X,d_X) \to (Y,d_Y)$ is an isometry and $\mathcal{I}_2:A \to A$ is the identity \cite[Thm~3.1]{vayer2020fused}. In this case, we call the labelled mm-spaces $\Xf_1$, $\Xf_2$ \emph{isomorphic}.

Similarly to \cref{sec:gromov-wasserstein}, 
for $\eps \ge 0$,
the \emph{regularized, unbalanced FGW transport} may be defined by
\begin{align*}
\UFGW_{\beta,\eps}(\Xf_1,\Xf_2)
\coloneqq
\smashoperator[l]{\inf_{\pi \in \M^+ (X_1,X_2)}}
 &\Bigl(
 \int_{(X_1 \times A \times X_2 \times A)^2} 
 (1 - \beta) \,
 |d_1(x_1,x_1') - d_2(x_2,x_2')|^2
 \\
 &+ \tfrac\beta2 \,
  e^2(a_1,a_2)
 + \tfrac\beta2 \,
  e^2(a_1',a_2')
  \dx \pi
  \dx \pi
  \\
  &
  +\sum_{i=1}^2 D^\otimes_{\phi_i} (\pi_i, \mu_i)
  + \eps \KL^\otimes (\pi, \mu_1 \otimes \mu_2)
  \Bigr)^{\frac12}
\end{align*}
with $\pi \coloneqq (P_i)_\# \pi$, 
where $P_i((x_1,a_1),(x_2,a_2)) \coloneqq (x_i,a_i)$, $i=1,2$.

For the multi-marginal generalization,
we consider the labelled mm-spaces $\Xf_i \coloneqq (X_i \times A,d_i, e,\mu_i)$, 
$i =1,\dotsc,N$,
defined over the same label space $(A,e)$
and employ the Cartesian products
\begin{equation*}
    \vA \coloneqq A^N 
    \quad\text{and}\quad
    \vXA \coloneqq
    \bigtimes_{i=1}^N (X_i \times A).
\end{equation*}
For convenience,
we use the notation $(\vx,\va) \in \vXA$,
where $\vx \coloneqq (x_1, \dots, x_n) \in \vX$ contain the structure parts
and $\va \coloneqq (a_1, \dots, a_N) \in \vA$ the label parts,
and where $(x_i, a_i)$ correspond to the $i$th factor $(X_i \times A)$ of $\vXA$.
We introduce the \emph{(regularized, unbalanced) multi-marginal FGW transport} problem by
\begin{equation*}
\MFGW_\beta(\Xf_1,\dotsc,\Xf_N) 
\coloneqq
\smashoperator[l]{\inf_{\pi \in \Pi(\mu_1,\dotsc,\mu_N)}} 
\int_{(\vXA)^2}
\!\!\!\!\!
(1-\beta) \, c_{\struc}(\vx,\vx')
%\\
%&
+ \tfrac\beta2 \, c_{\lab}(\va)
+ \tfrac\beta2 \, c_{\lab}(\va')
\dx \pi(\vx,\va) \dx \pi(\vx',\va')
\end{equation*}
and
\begin{align*}
\UMFGW_{\beta,\eps}(\Xf_1,\dotsc,\Xf_N) 
\coloneqq
\inf_{\pi \in \M^+(\vXA)}
&\int_{(\vXA)^2}
(1-\beta) \, c_{\struc} (\vx,\vx') 
+ \tfrac\beta2 \, c_{\lab}(\va)
+ \tfrac\beta2 \, c_{\lab}(\va')\\
&\dx \pi(\vx,\va) \dx \pi(\vx',\va')
+ \sum_{i=1}^N D_{\phi_i}^\otimes(\pi_i, \mu_i)
    + \eps \KL^\otimes(\pi,\mu^\otimes),
\end{align*}
where $c_{\struc} \colon \vX \times \vX \to [0, \infty]$
is the cost function on the structure
and $c_{\lab}\colon \vA \to [0, \infty]$
on the labels.

\begin{remark}
\label{rem:fgw-theory}
The fused transport problems MFGW and UMFGW are, 
in fact, 
special cases of MGW and UMGW 
with respect to the cost function
\begin{equation}
    \label{eq:c-fused}
    c_{\fused} ((\vx,\va), (\vx',\va')) 
    \coloneqq
    (1-\beta) \, c_{\struc}(\vx,\vx') + \tfrac\beta2 \, c_{\lab}(\va) + \tfrac\beta2 \, c_{\lab}(\va')
\end{equation}
and where the mm-spaces $\XX_1, \dots, \XX_N$ are replaced by the 
compact pseudometric measure spaces $\Xf_1, \dots, \Xf_N$.
Since the definiteness of the metric is not exploited in the proofs, 
the existence of an optimal multi-marginal transport in \cref{prop:existence},
the tightness of the bi-convex relaxation in \cref{thm:tight},
and the alternating minimization procedure in \cref{alg:UMGW,prop:relax_as_ot}
remain valid.
The function $c_{\fused} \colon (\vXA)^2 \to [0,\infty]$ is marginal conditionally negative definite
if and only if 
$c_{\struc} \colon \vX \times \vX \to [0, \infty]$ is marginal conditionally negative definite
because the sums related to $c_{\lab}$ vanish in \cref{mcd}.
\end{remark}

Naturally, the formulation gives rise to fused Gromov--Wasserstein barycenters.
To the best of our knowledge, the only existing algorithms for such problems are based on conditional gradient descent as proposed in \cite{vayer2020fused}.
In the following, we discuss the relation between (U)MFGW and (U)FGW barycenters,
which also leads to a new algorithm for their computation.
For $\rho \coloneqq (\rho_i)_{i=1}^N$ with $\rho_i \ge 0$ and $\sum_{i=1}^N \rho_i = 1$,  
we define the weighted mean function $m_\rho\colon (\R^d)^N \to \R^d$ by
\begin{equation*}
    m_\rho(\va) =
    m_\rho (a_1,\dotsc,a_N)
    \coloneqq
    \sum_{i=1}^N \rho_i \, a_i
    = \argmin_{b\in \R^d} \sum_{i=1}^N \rho_i\|a_i - b\|^2,
\end{equation*}
where $\|\cdot\|$ denotes the Euclidean norm.
Analogously to \cref{eq:gw_bary},
the \emph{fused Gromov--Wasserstein barycenter} between 
the labelled mm-spaces $\Xf_i \coloneqq (X_i \times A, d_i, e, \mu_i)$,
$i=1,\dots,N$, 
on $(A,e)$ is a minimizing labelled mm-space $\Yf \coloneqq (Y \times A, d, e, \nu)$ of
\begin{equation}\label{eq:fgw_bary}
  \inf_{\Yf} \sum_{i=1}^N \rho_i \FGW^2_\beta(\Xf_i,\Yf).
\end{equation}

\begin{theorem}[Free-Support Fused Barycenter]
\label{thm:fus-gen-bary}
  Let $\Xf_i$, $i=1,\dots, N$, be labelled mm-spaces on $(A,e)$,
  where $A \subset \R^d$ is a convex compact set 
  and $e(a,a') \coloneqq \|a - a'\|$ is the Euclidean distance.
  Let $\rho_i \ge 0$ be weights with $\sum_{i=1}^N \rho_i = 1$. Then, the infimum in \cref{eq:fgw_bary} is attained. Moreover, each solution of \cref{eq:fgw_bary} is isomorphic to a labelled mm-space of the form $\Yf^* = (\vX \times A, d^*,e,\nu^*)$ with
  \begin{equation*}
  d^*(\vx,\vx') \coloneqq \sum_{i=1}^N \rho_i d_i(x_i,x_i'),
  \quad
  \nu^* \coloneqq (P_{\vX},m_\rho)_\# \pi^*,
  \end{equation*}
  where $\pi^*$ is a minimizer of $\MFGW(\Xf_1, \dots, \Xf_N)$
  with costs
  \begin{align*}
    c_{\struc}(\vx,\vx')
    &\coloneqq
    \frac12
    \sum_{i,j=1}^N
    \rho_i \rho_j \,
    \lvert d_i(x_i,x_i') - d_j(x_j,x_j') \rvert^2,
    \\
    c_{\lab}(\va) 
    &\coloneqq
    \frac12
    \sum_{i,j=1}^N
    \rho_i \rho_j \,
    \| a_i - a_j \|^2,
  \end{align*}
  and where
  $(P_{\vX}, m_\rho) \colon \vXA \to (\vX \times A)
  : (\vx,\va) \mapsto (\vx, m_\rho(\va))$.
\end{theorem}

\begin{proof}
  The major difference to the free-support barycenter in \cref{eq:gw_bary}
  is the additional label cost $c_{\lab}$,
  which may be rearranged as
  \begin{equation*}
      c_{\lab}(\va)
      = \sum_{i=1}^N \rho_i \, \| a_i \|^2 
      - \sum_{i,j=1}^N \rho_i \rho_j \, \langle a_i, a_j \rangle
      = \sum_{i=1}^N \rho_i \, \Bigl\| a_i - \sum_{j=1}^N \rho_j a_j \Bigr\|^2
      = \sum_{i=1}^N \rho_i \, \| a_i - m_\rho(\va) \|^2.
  \end{equation*}
  The statement can be established similarly to the proof of \cref{thm:gen-bary}.
  For this, 
  let $\Yf = (Y \times A,d,e,\nu)$
  be some labelled mm-space on $(A,e)$.
  Further,
  let $\smash{\pi^{(i)}_{\FGW}}$ be optimal with respect to $\FGW_\beta(\Xf_i,\Yf)$, $i=1,\dotsc,N$,
  and $\pi_{\mathrm{g}} \in \Pi(\mu_1,\dotsc,\mu_N,\nu)$ be the associated gluing along $(Y \times A)$.
  Exploiting that 
  the Euclidean mean $m_\rho(\va)$ minimizes $\min_{b \in \R^d} \sum_{i=1}^N \rho_i \| a_i - b \|^2$
  for all $\va \in \vA$,
  and $d^*(\vx,\vx')$ minimizes
  \smash{$\min_{t \in \R} \sum_{i=1}^N \rho_i \lvert d_i(x_i,x_i') - t \rvert^2$}
  for all $\vx \in \vX$,
  we obtain 
  similarly as in \cref{eq:gw-mgw} that
  
  \begin{align}
        &\sum_{i=1}^N \rho_i \FGW^2_\beta(\Xf_i,\Yf)\\
        =
         &\int_{(\vXA \times (Y \times A))^2} 
         (1-\beta) 
         \sum_{i=1}^N \rho_i \,
         |d_i(x_i,x_i') - d(y, y')|^2
         \notag \\
         &\qquad
         + \tfrac\beta2 
         \sum_{i=1}^N \rho_i \,
          \lVert a_i - b \rVert^2
          + \tfrac\beta2 
         \sum_{i=1}^N \rho_i \,
          \lVert a_i' - b' \rVert^2
          \dx \pi_{\mathrm g}((\vx,\va),(y,b))
          \dx \pi_{\mathrm g}((\vx,\va),(y,b))
        \notag \\
        \ge
         &\int_{(\vXA \times (Y \times A))^2} 
         (1-\beta) 
         \sum_{i=1}^N \rho_i \,
         |d_i(x_i,x_i') - d^*(\vx, \vx')|^2
         \notag \\
         &\qquad
         + \tfrac\beta2 
         \sum_{i=1}^N \rho_i \,
          \lVert a_i - m_\rho(\va) \rVert^2
          + \tfrac\beta2 
         \sum_{i=1}^N \rho_i \,
          \lVert a_i' - m_\rho(\va') \rVert^2
          \dx \pi_{\mathrm g}
          \dx \pi_{\mathrm g}
          \notag \\
          \ge &\MFGW_\beta(\Xf_1, \dots, \Xf_N).
    \label{eq:gw-mfgw}
    \end{align}%
    Based on
  the substitution 
  $\pi^{(i)}_* \coloneqq (P_i, P_{\vX}, m_\rho)_\# \pi^*$
  with 
  \[
  (P_i,P_{\vX},m_\rho)(\vx, \va) \coloneqq ((x_i,a_i), \vx, m_\rho(\va))
  \]
  for any optimal $\pi^* \in \p(\vXA)$ of $\MFGW_\beta(\Xf_1,\dotsc,\Xf_N)$, 
  we conclude similarly as in \cref{eq:mgw-gw} that
    \begin{align}
    &\MFGW_\beta(\Xf_1,\dotsc,\Xf_N)
    \notag \\
    &=
    \sum_{i=1}^N \rho_i \biggl(
         \int_{(X_i \times A \times \vX \times A)^2}
        (1-\beta) |d_i(x_i, x_i') - d^*(\vy,\vy')|^2 
    \notag \\
    &\qquad
    + \tfrac\beta2  \|a_i - b\|^2
    + \tfrac\beta2  \|a_i' - b'\|^2
      \dx \pi^{(i)}_{*}((x_i,a_i),(\vy,b))
      \dx \pi^{(i)}_{*}((x_i',a_i'),(\vy',b')) \biggr)
      \notag \\
      &\geq
      \sum_{i=1}^N \rho_i \FGW^2_\beta(\Xf_i, \Yf^*).
    \label{eq:mfgw-fgw}
  \end{align}
  Due to \cref{eq:gw-mfgw},
  the last inequality has to be an equality
  and
  $\Yf^*$ is a barycenter.

    Vice versa \cref{eq:mfgw-fgw} guarantees 
    equality in \cref{eq:gw-mfgw} for
    every fused barycenter 
    $\Yf^\dagger \coloneqq (Y^\dagger \times A, d^\dagger,e,\nu^\dagger)$.
    On the basis of the optimal plans $\smash{\pi^{(i)}_{\FGW}}$
    between $\Xf_i$ and $\Yf^\dagger$,
    we build a gluing
    $\pi_{\mathrm{g}}^\dagger \in \Pi(\mu_1,\dotsc,\mu_N,\nu^\dagger)$
    as above.
    Noticing that
    $(P_{\vXA})_\# \pi_{\mathrm{g}} \in \Pi(\mu_1,\dotsc,\mu_N)$ 
    is a solution to 
    $\MFGW_\beta(\Xf_1,\dotsc,\Xf_N)$,
    we going to show that
    $\Yf^\dagger$ is isomorphic to
    $\Yf^* \coloneqq (\vX \times A, d^*, e, \nu^*)$
    with $\nu^* \coloneqq (P_{\vX},m_\rho)_\# (P_{\vXA})_\# \pi_{\mathrm g}^\dagger$.
    Similarly to the last part of the proof of \cref{thm:gen-bary},
    the equality in \cref{eq:gw-mfgw} with respect to $\Yf^\dagger$,
    the non-negativity of the integrands, 
    and the minimizing properties of $m_\rho$ and $d^*$
    imply that
  \begin{align*}
        &m_\rho(\va) = b ,
        \quad
        m_\rho(\va') = b', 
        \quad\text{and}\quad
        d^*(\vx,\vx') = d^\dagger(y,y') 
        \\
        &\text{for } 
        (\pi_{\mathrm{g}}^\dagger \otimes \pi_{\mathrm{g}}^\dagger)
        \text{-a.e.\ } ((\vx,\va),(y,b),(\vx',\va'),(y',b')).
  \end{align*}
  Considering
  $\gamma \coloneqq {(P_{\vX},m_\rho,P_{Y \times A})}_\# \pi_{\mathrm{g}} 
  \in \Pi(\nu^*, \nu^\dagger)$,
  where 
  \begin{equation*}
      (P_{\vX},m_\rho,P_{Y \times A})((\vx,\va),(y,b)) 
      \coloneqq ((\vx, m_\rho(\va)),(y,b)),    
  \end{equation*}
  we finally have
  \begin{align*}
      \FGW_{\beta}^2(\Yf^*,\Yf^\dagger) 
      &\leq \int_{((\vX \times A) \times (Y \times A))^2} 
 (1-\beta) \, |d^*(\vx,\vx') - d^\dagger(y,y')|^2
 + \tfrac\beta2 \,
  e^2(b_1,b_2)
  + \tfrac\beta2 \,
  e^2(b_1',b_2')\\
  &\qquad
  \dx \gamma((\vx,b_1),(y,b_2))
  \dx \gamma((\vx',b_1'),(y',b_2'))\\
  &= \int_{(\vXA \times (Y \times A))^2} 
 (1-\beta) \,
 \lvert
 \underbrace{d^*(\vx,\vx') - d^\dagger(y,y')}_{\equiv 0 \; \text{a.e.}}
 \rvert^2
 + \tfrac\beta2 \,
  \lVert
  \underbrace{m_\rho(\va) - b}_{\equiv 0 \; \text{a.e.}}
  \rVert^2\\
  &\qquad
  + \tfrac\beta2 \,
  \lVert
  \underbrace{m_\rho(\va') - b'}_{\equiv 0 \; \text{a.e.}}
  \rVert^2
  \dx \pi_g^\dagger((\vx,\va),(y,b))
  \dx \pi_g^\dagger((\vx',\va'),(y',b'))
  = 0,
  \end{align*}
  which concludes the proof.
\end{proof}

Due to the difficult interpretation of the free-support barycenter \cref{eq:fgw_bary}
in the context of imaging,
we restrict our attention to
\emph{fixed-support unbalanced fused barycenters} which are solutions of
\begin{equation}\label{eq:fugw_bary}
\min_{\nu \in \M^+(Y \times A)} \sum_{i=1}^N \rho_i \UFGW_{\beta,\eps}(\Xf_i,\Yf) \quad \text{with} \quad \Yf = (Y \times A, d, e, \nu),
\end{equation}
where the metric space $(Y,d)$ is given in advance.
In analogy to \cref{eq:rest-bary},
UFGW may be unbalanced in $\Xf_i$ with respect to some entropy function $\phi_i$ 
and balanced in $\Yf$. Similarly to before we use $\bullet$ as a dummy measure for labelled mm-spaces whenever the associated problems are independent of the measure.

\begin{theorem}[Fixed-Support Fused Barycenter]
  \label{thm:fus-rest-bary}
  Let $\Xf_i \coloneqq (X_i \times A, d_i, e, \mu_i)$ be labelled mm-spaces on $(A,e)$,
  and $\rho_i \ge 0$ be weights
  with $\sum_{i=1}^N \rho_i = 1$. The infimum in \cref{eq:fugw_bary} is attained for $\Yf^* \coloneqq (Y \times A, d, e, \nu^*)$
  with $\nu^* \coloneqq (P_{Y \times A})_\# \pi^*$,
  where $\pi^*$ solves
  $\UMFGW_\beta(\Xf_1,\dots,\Xf_N, (Y \times A,d,e,\bullet))$
  with cost functions
  \begin{align*}
    c_{\struc}((\vx,y),(\vx',y'))
    &\coloneqq
    \sum_{i=1}^N \rho_i \,
    | d_i(x_i,x_i') - d(y,y')|^2,
    \\
    c_{\lab}(\va,b) 
    &\coloneqq 
    \sum_{i=1}^N \rho_i \, e(a_i,b).
  \end{align*}
  The unregularized UMFGW ($\eps=0$) is unbalanced in $\Xf_i$
  with respect to $\rho_i\phi_i$
  and unrestricted in $\Yf$. Furthermore, for every barycenter
  $\YY^\dagger = (Y \times A,d,e,\nu^\dagger)$,
  there exists a multi-marginal solution $\pi^*$ 
  of $\UMFGW(\Xf_1,\dotsc,\Xf_N,(Y \times A,d,e,\bullet))$ 
  so that $(P_{(Y \times A)})_\# \pi^* = \nu^\dagger$.
\end{theorem}

\begin{proof}
  The statement can be established using a similar argumentation
  as in the proof of \cref{thm:rest-bary}, 
  where the cost function $c$  has to be replaced by $c_{\fused}$.
\end{proof}

\begin{remark}[Fused Barycenter Algorithm]
    Due to \cref{thm:fus-rest-bary}, 
    the unbalanced fixed-support barycenter can be numerically computed
    using two steps.
    First, 
    the minimizer $\pi^*$ of 
    $\UMFGW_\beta(\Xf_1,\dots,\Xf_N, (Y \times A,d,e,\bullet))$
    is calculated
    by applying \cref{alg:UMGW} 
    with respect to the special cost function $c_{\fused}$ in \cref{eq:c-fused},
    see also \cref{rem:fgw-theory}.
    Secondly, 
    we compute the marginal
    $\nu^* \coloneqq (P_{(Y \times A)})_\# \pi^*$.
\end{remark}

%-----------------------------------------------------
\section{Numerical Examples} \label{sec:numerics}
%-----------------------------------------------------
The multi-marginal GW transport
is especially useful for image processing tasks
like the computation of barycenters, progressive interpolation
and noise removal.
We consider (gray-value) images on box domains
which may be described as piecewise constant functions on their pixel grid.
These images can be interpreted as discrete mm-spaces $\XX \coloneqq (X, d, \mu)$,
where $X$ is the grid containing the centres of the pixels,
$d$ is the normalized Euclidean distance, and
$\mu$ is the probability measure
corresponding to the (normalized) gray values.
For all numerical computations,
we replace $\mu^\otimes$ in UMGW and UMFGW by the counting measure
so that the regularizer corresponds to the discrete entropy.
The experiments are performed on an off-the-shelf MacBook Pro (Apple M1 chip, 8~GB RAM) and are implemented\footnote{The source code is publicly available at \url{https://github.com/Gorgotha/UMGW}.} in Python 3. For the numerical computations, 
we partly rely on the Python Optimal Transport (POT) toolbox \cite{POT-toolbox} as well as some publicized implementations of \cite{SejViaPey21}. Some of the following input data is taken from \cite{shapes2d}.

\begin{figure}
    \centering
    \includegraphics[width=130pt]{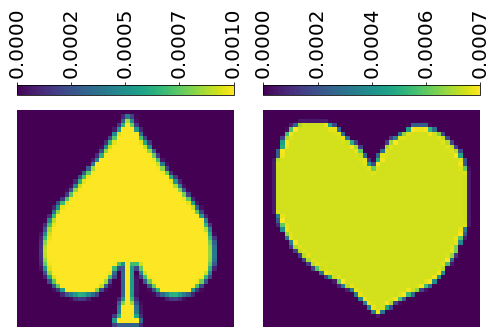}
    \caption{The ($50 \times 50$)-pixel test images of the numerical experiment in \cref{sec:bal-bary-2d}. (Left) The spade $\XX_1$. (Right) The heart $\XX_2$.}
    \label{fig:ex1_1}
\end{figure}

\subsection{Balanced 2D Barycenters
and Run-Time Comparison}
\label{sec:bal-bary-2d}

\begin{figure}
     \centering
     \begin{subfigure}[b]{0.5\textwidth}
         \centering
         \includegraphics[width=\linewidth]{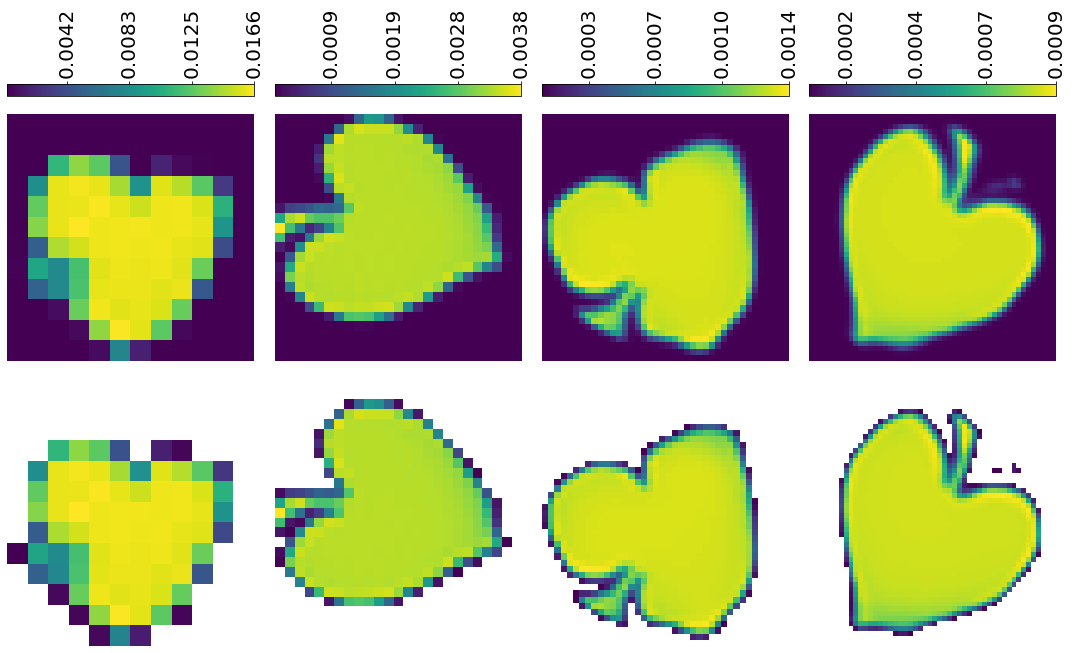}
         \newline
         \footnotesize
         \noindent
         $12 \times 12$ \hspace{20pt}
         $25 \times 25$ \hspace{20pt}
         $40 \times 40$ \hspace{20pt}
         $50 \times 50$
         \caption{Fixed-support barycenters computed using UMGW
         for different grid sizes. 
         (Top) Computed measure on the full support.
         (Bottom) Computed measure restricted to the essential support.}
         \label{fig:ex1_2}
     \end{subfigure}
     %\hfill
     \begin{subfigure}[b]{0.48\textwidth}
         \centering
         \rotatebox{90}{
         \tiny %\hspace{-1pt}
         RBCD ($\eps_1$) \hspace{6pt}
         RBCD ($\eps_2$) \hspace{9pt}
         BCD}%
         \includegraphics[width=0.97\linewidth]{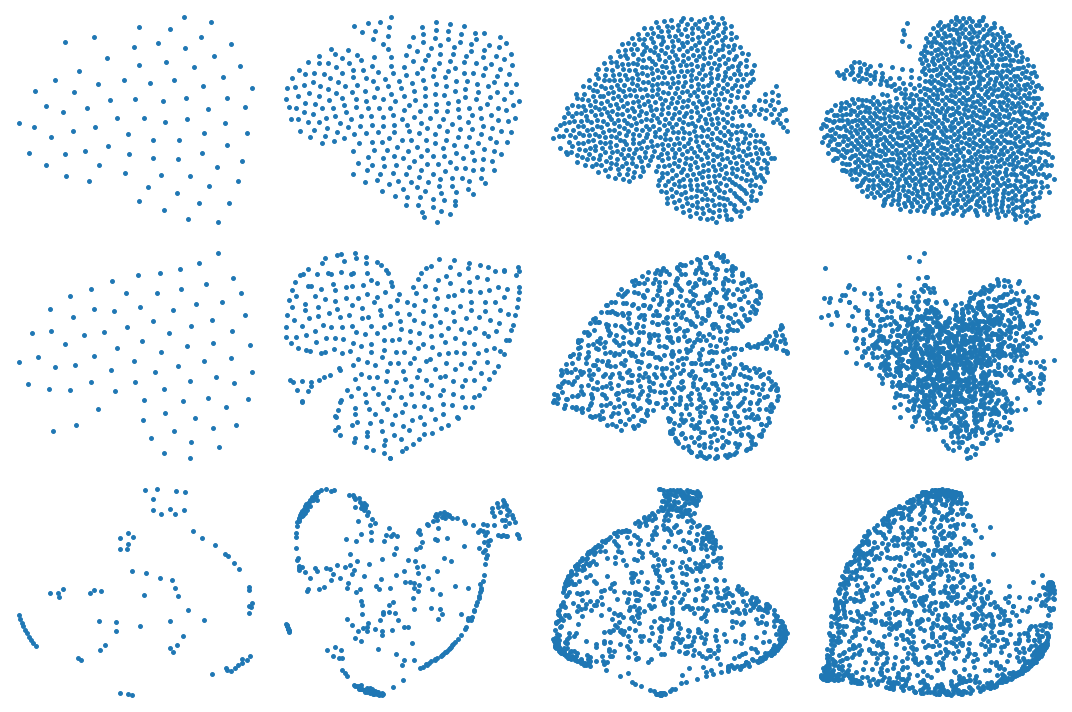}
         \newline
         \footnotesize
         \hspace*{9pt}
         $81$ \hspace{32pt}
         $348$ \hspace{32pt}
         $901$ \hspace{32pt}
         $1379$
         \caption{Free-support barycenters computed using (R)BCD 
         for different numbers of support point.
         The regularization parameter is chosen as 
         $\eps_1 \coloneqq 5 \cdot 10^{-3}$
         and $\eps_2 \coloneqq 5 \cdot 10^{-4}$.}
         \label{fig:ex1_3}
     \end{subfigure}
     \caption{Comparison between the balanced barycenters 
     between the images in \cref{fig:ex1_1} 
     computed using 
     UMGW (\cref{rem:umgw_for_bary}) 
     and (R)BCD (\cref{rem:peyre}).
     The UMGW barycenters are determined on full square grids 
     of different sizes.
     The (R)BCD barycenters are computed using differed fixed numbers of support 
     points and are embedded in 2d using MDS. 
     The number of support points corresponds to the number of essential 
     support points in the computed UMGW barycenters.}
     \label{fig:ex1}
\end{figure}

In this first example,
we compare the UMGW method 
for fixed-support barycenters in \cref{rem:umgw_for_bary}
with block coordinate descent (BCD) 
and its regularized version (RBCD) 
proposed by Peyré, Cuturi, and Solomon \cite{PCS2016}.
A brief summary of (R)BCD is given in \cref{rem:peyre}. 
As mentioned, our method requires us to fix the support beforehand
and yields a measure on this support,
whereas BCD and RBCD fix a discrete measure on a fixed number of support points
and yield a dissimilarity matrix. For the latter procedures, we employ the implementations of the POT library \cite{POT-toolbox}.
For the numeric computations,
we consider two $50 \times 50$ pixel images of a spade and a heart, see \cref{fig:ex1_1}, from which we extract the mm-spaces $\XX_1$ and 
$\XX_2$ with normalized Euclidean distances. 
Using the UMGW method in \cref{rem:umgw_for_bary} 
with $\eps = 0.15 \cdot 10^{-3}$, $\rho_1\coloneqq\rho_2 \coloneqq 0.5$,
we compute several balanced barycenters on different grids in the image domain
consisting of $12 \times 12$, $25 \times 25$, $40 \times 40$, and $50 \times 50$ pixels.
The numerical computations are restricted to 200 seconds,
and the resulting barycenters are shown in \cref{fig:ex1_2}.

\begin{figure}
    \centering
    \small
    \hfill
    \raisebox{2pt}{\tikz{\draw[thick, PLOTblue] (0pt,0pt) -- (15pt,0pt);}}
    UMGW, \;
    \raisebox{2pt}{\tikz{\draw[thick, PLOTorange] (0pt,0pt) -- (15pt,0pt);}}
    BCD, \;
    \raisebox{2pt}{\tikz{\draw[thick, PLOTgreen] (0pt,0pt) -- (15pt,0pt);}}
    RBCD ($\eps_2$), \;
    \raisebox{2pt}{\tikz{\draw[thick, PLOTgreen, densely dashed] (0pt,0pt) -- (15pt,0pt);}}
    RBCD ($\eps_1$)
    \hspace{12pt}
    \newline
    \rotatebox{90}{\hspace{20pt}\footnotesize{barycentric loss}}
    \includegraphics[width=400pt]{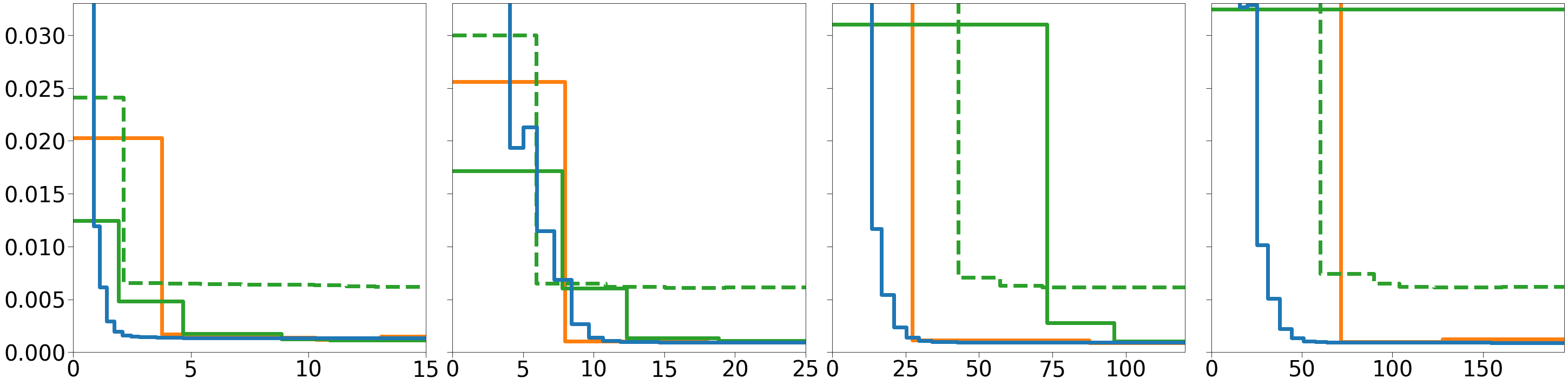}
    \\[-3pt]
     \hspace*{25pt}
     \footnotesize
     $12 \times 12$ grid, \hspace{38pt}
     $25 \times 25$ grid, \hspace{38pt}
     $40 \times 40$ grid, \hspace{38pt}
     $50 \times 50$ grid, \\
     \hspace*{22pt}
     $81$ points \hspace{51pt}
     $348$ points \hspace{49pt}
     $901$ points \hspace{46pt}
     $1379$ points
    \caption{Barycentric loss 
    of the UMGW method and the (R)BCD method 
    to compute the balanced barycenter between the images in \cref{fig:ex1_1}
    plotted against computation time in seconds.
    The loss is shows with respect to different grid sizes 
    and numbers of support-points.
    }
    \label{fig:2re}
\end{figure}

\begin{figure}
    \centering
    \includegraphics[width=300pt]{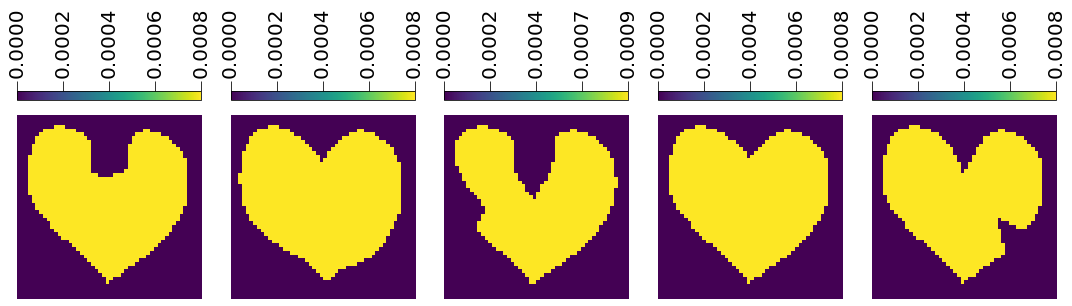}
    \caption{The ($50 \times 50$)-pixel test images of the multi-input experiment
    in \cref{sec:bal-bary-2d}.}
    \label{fig:multi-marginals_1}
\end{figure}

\begin{figure}
    \centering
    \small
    \hfill
    \raisebox{2pt}{\tikz{\draw[thick, PLOTblue] (0pt,0pt) -- (15pt,0pt);}}
    UMGW, \;
    \raisebox{2pt}{\tikz{\draw[thick, PLOTorange] (0pt,0pt) -- (15pt,0pt);}}
    BCD, \;
    \raisebox{2pt}{\tikz{\draw[thick, PLOTgreen] (0pt,0pt) -- (15pt,0pt);}}
    RBCD ($\eps_2$)
    \hspace{17pt}
    \newline
    \rotatebox{90}{\hspace{25pt}barycentric loss}
    \includegraphics[width=380pt]{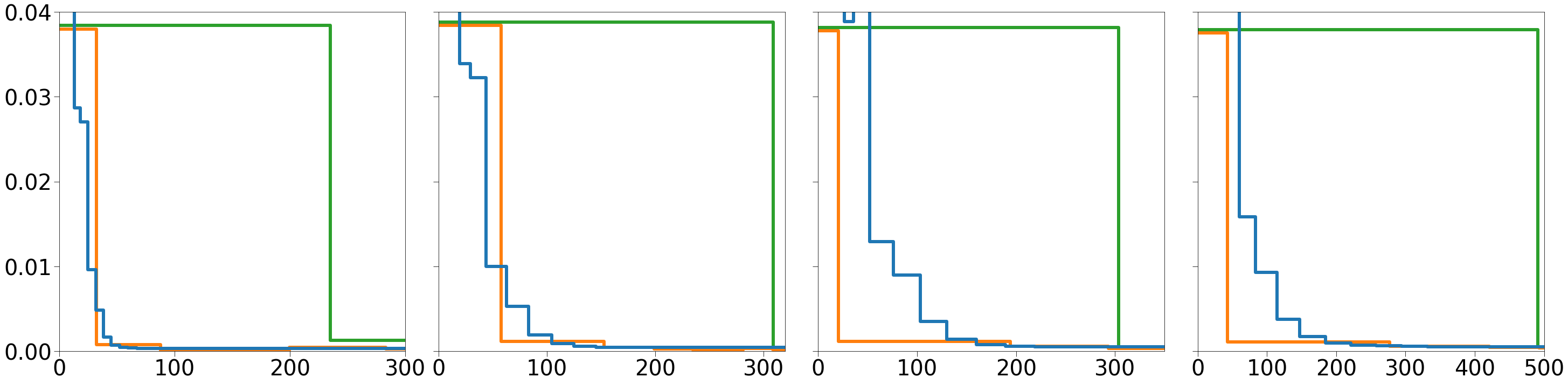}
    \\[-3pt]
     \hspace*{20pt}
     2 inputs \hspace{50pt}
     3 inputs \hspace{45pt}
     4 inputs \hspace{50pt}
     5 inputs
    \caption{Barycentric loss 
    of the UMGW method and the (R)BCD method 
    to compute the balanced barycenter between 
    several numbers of inputs from \cref{fig:ex1_1}
    plotted against computation time in seconds.
    }
    \label{fig:multi-marginals_2}
\end{figure}
    
We proceed with with computing (R)BCD balanced barycenters,
see \cref{rem:peyre},
where we initalize the (R)BCD method
by a uniform distribution on a fixed number of support point.
For a fair comparison,
the number of support points is chosen with respect to
the number of essential support points of the UMGW barycenters in \cref{fig:ex1_2},
that is,
to the number of pixels with mass greater than $10^{-4}$.
Moreover, the computation time is again restricted to 200 seconds.
For the regularized variant,
we consider $\eps_1 \coloneqq 5 \cdot 10^{-3}$ and $\eps_2 \coloneqq 5 \cdot 10^{-4}$.
Unfortunately, lower choices result in numerical errors.
The outputs consist of dissimilarity matrices,
which are embedded in two dimensions
using multi-dimensional scaling (MDS).
The embedded support-free barycenters are shown in \cref{fig:ex1_3}.

Due to the rotational and reflectional invariance of GW, 
the computed barycenters in \cref{fig:ex1} do not follow any particular alignment.
Moreover,
the GW barycenter does not have to be unique.
Despite the different nature of the barycenters (image or point cloud),
the results for our UMGW-based method and BCD are qualitatively comparable.
Considering the evolution of
the barycentric loss $(\GW^2_\eps(\XX_1, \YY) + \GW^2_\eps(\XX_2, \YY))/2$
over the computation time,
our UMGW method perform similar well than (R)BCD,
which substantiate the discussion in \cref{rem:peyre}
conjecturing that the complexity of both method  is comparable.
Increasing the number of inputs,
see \cref{fig:multi-marginals_1},
which represent $50 \times 50$ pixel images from the heart class in \cite{shapes2d},
we repeat the experiment to determine a balanced barycenter 
on a $50 \times 50$ pixel grid using UMGW.
For (R)BCD, 
the number of points is again chosen with respect to
the number of essential pixels of the UMGW barycenter.
\Cref{fig:multi-marginals_2} shows the respective barycenter losses against time in seconds. Intriguingly, while BCD scales very well,
and UMGW only slightly worse than BCD,
the computation of RBCD increases dramatically.

\subsection{Progressive Interpolation on Non-Euclidean Domains}\label{sec:progr-interp}
In this example, 
we utilize the UMGW-based barycenter method in \cref{rem:umgw_for_bary}
to calculate a progressive GW interpolation 
between measures on the unit sphere in $\R^3$.
Using the standard spherical coordinates, 
we discretize the unit sphere by an $80 \times 80$ pixel grid on $[0,2\pi] \times [0,\pi]$
and endow it with the normalized great circle distance. 
In this manner, 
we obtain a (discrete) metric space $(X,d_X)$. 
Next,
we equip $(X,d_X)$ with measures representing the landmasses of the earth 
at two points in time. 
Here, $\mu_1$ corresponds to the distribution of \emph{Pangea}, 
a supercontinent containing the entire landmass of the earth until the jurassic period.
On the other hand,  
$\mu_2$ corresponds to the distribution of the today landmasses.
These measures lead us to the mm-spaces 
$\XX_1 = (X,d_X,\mu_1)$ and $\XX_2 = (X,d_X,\mu_2)$. 
A visualization is given 
on the uttermost left-hand and right-hand side
of \cref{fig:prog_interpolation}. 
Note that 
the point masses around the poles become smaller
to respect the denser discretization.
Choosing $\eps = 3 \cdot 10^{-3}$,
we compute the balanced barycenters
for $(\rho_1,\rho_2) = (0.8,0.2),(0.6,0.4),(0.4,0.6),(0.2,0.8)$.
These barycenters are also shown in \cref{fig:prog_interpolation}.
Obtaining such an interpolation with BCD or RBCD is non-trivial as
the mm-space corresponding to the resulting dissimilarity matrix 
needs to be embedded on the sphere,
which poses an additional challenge.
Clearly,
the approach in this example can be applied 
to compute GW barycenters on arbitrary compact Riemannian manifolds.

\setlength{\labwidth}{59.0pt}
\begin{figure}
    \centering
    \includegraphics[width = 0.95\linewidth]{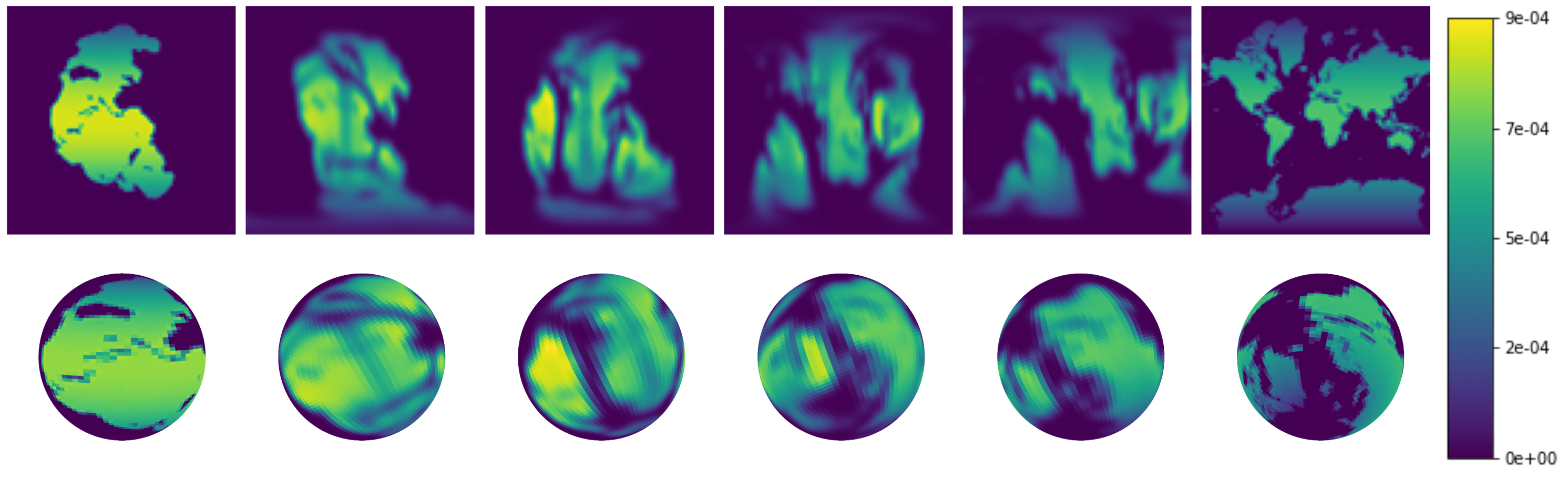}\\[-10pt]
    \scriptsize
    \hspace*{-27pt}
    \parbox{\labwidth}{\centering $\XX_1$}
    \parbox{\labwidth}{\centering $\YY_1$}
    \parbox{\labwidth}{\centering $\YY_2$}
    \parbox{\labwidth}{\centering $\YY_3$}
    \parbox{\labwidth}{\centering $\YY_4$}
    \parbox{\labwidth}{\centering $\XX_2$}
    \caption{Progressive Gromov--Wasserstein interpolation
      via multi-marginal Gromov--Wasserstein transport in the balanced setting.}
    \label{fig:prog_interpolation}
\end{figure}

\subsection{Fused Gromov--Wasserstein Barycenter of Labelled Images}

In this section, we turn our attention to the computation of barycenters of shapes coming from 2D images which pixels admit a unique label. The following examples indicate that the fused version of our barycenter procedure yields desirable results when incorporating this additional label information.

For the {\bfseries first example},
we consider the MNIST database \cite{LBBH98}, 
which consists of handwritten digits on a ($28 \times 28$)-pixel grid. 
By concatenation and rotation of digits,
we construct four ($30 \times 34$)-pixel images
containing the numbers 98 and 81,
see the first two columns of \cref{fig:98_fused,fig:81_fused}.
The images are again interpreted as mm-spaces.
Furthermore,
we label the different digits of the numbers.
For this,
we use the label space $(A,g)$
where $A \coloneqq \{0,1\}$ 
and $e$ is the Euclidean distance,
and annotate the pixels of the first digit by $0$ and of the second digit by $1$.
The (labelled) mm-spaces are shown in first two columns of \cref{fig:98_fused,fig:81_fused}. Here blue corresponds to label $0$ and red to label $1$. 

\setlength{\labwidth}{67pt}
\begin{figure}
    \centering
    \includegraphics[width=0.95\linewidth]{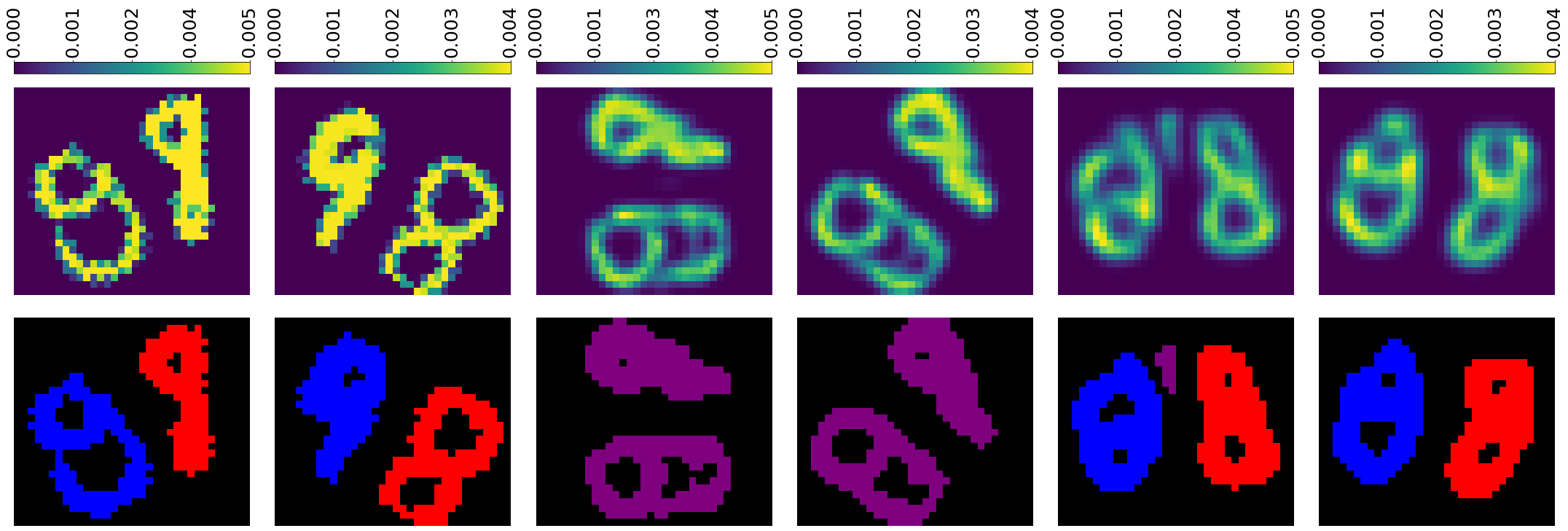}\\
    \scriptsize
    \hspace*{0pt}
    \parbox{\labwidth}{\centering $\XX_1$/$\Xf_1$}
    \parbox{\labwidth}{\centering $\XX_2$/$\Xf_2$}
    \parbox{\labwidth}{\centering GW}
    \parbox{\labwidth}{\centering UGW}
    \parbox{\labwidth}{\centering FGW}
    \parbox{\labwidth}{\centering FUGW}
    \caption{First row, left to right: input marginals, GW barycenter, UGW barycenter, FGW barycenter, UFGW barycenter. 
    Second row, left to right: labels of the input marginals, 
    visualized transport to the barycenters
    by averaging the label colour of the incoming mass. 
    }
    \label{fig:98_fused}
\end{figure}

\begin{figure}
    \centering
    \includegraphics[width=0.95\linewidth]{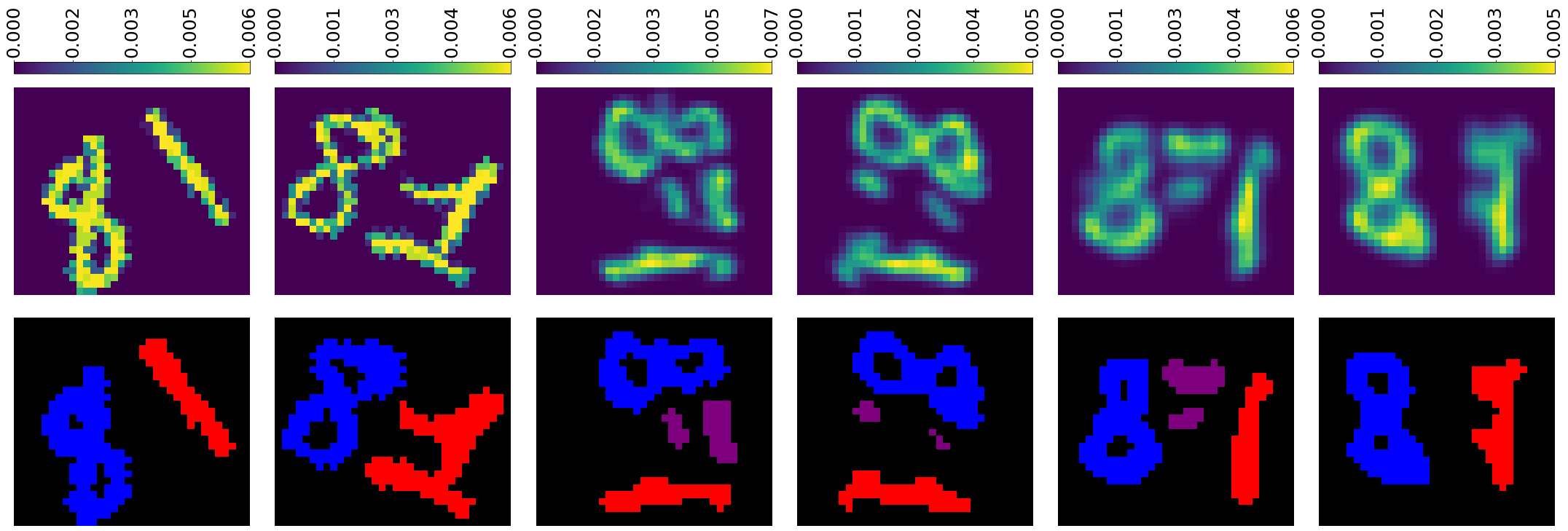}\\
    \scriptsize
    \hspace*{0pt}
    \parbox{\labwidth}{\centering $\XX_1$/$\Xf_1$}
    \parbox{\labwidth}{\centering $\XX_2$/$\Xf_2$}
    \parbox{\labwidth}{\centering GW}
    \parbox{\labwidth}{\centering UGW}
    \parbox{\labwidth}{\centering FGW}
    \parbox{\labwidth}{\centering FUGW}
    \caption{First row, left to right: input marginals, GW barycenter, UGW barycenter, FGW barycenter, UFGW barycenter. 
    Second row, left to right: labels of the input marginals, 
    visualized transport to the barycenters
    by averaging the label colour of the incoming mass.
    }
    \label{fig:81_fused}
\end{figure}

To determine fixed-support barycenters, 
we apply \cref{thm:rest-bary,thm:fus-rest-bary},
where the multi-marginal transport is computed by \cref{alg:UMGW}.
For this,
we set the support and distance of the mm-space $\YY \coloneqq (Y,d_Y,\bullet)$ according to a ($30 \times 34$)-pixel image. 
Moreover, 
we choose $\Yf \coloneqq ((Y_1 \times \{1\}) \cup (Y_2 \times \{0\}), d, e, \bullet)$, where $Y_1$ and $Y_2$ corresponds to the left and right half of the pixel grid, respectively.
We use the weights $\rho \coloneqq (0.5,0.5)$,
a small regulatization by $\eps \coloneqq 0.0002$,
the divergence $\phi_i \coloneqq 0.01 \, \phi_{\ent}$ for the unbalanced setting,
and the trade-off parameter $\beta = 0.5$ for the fused setting.
We compute the GW, UGW, FGW, and UFGW barycenters presented in the first row of \cref{fig:98_fused,fig:81_fused}. As usual with $\GW$, the non-fused barycenters do not follow any particular alignment. However, In the fused case the correct alignment is enforced by our choice of labelling in Y.
In the second row,
the transport to the barycenter is visualized 
by averaging the label colours with respect to 
the transported (labelled) pixels to a certain point in the barycenter.  

In \cref{fig:98_fused},
the shape of the eight in first space looks like the nine in the other space and vice versa. 
This is well reflected by the GW and UGW barycenters, 
which average the eight with the nine and the other way round.
Due to the additional labels in the fused setup,
the FGW and UFGW barycenter correspond to the matching of both nines and both eights.
In the balanced case, some mass has to be transported between both digits resulting in an undesired artifact.
Looking at the last barycenter, 
this can be avoided by using the unbalanced approach. 
In \cref{fig:81_fused}, 
both digits are matched in the correct order for all barycenters.
Due to the mass variations between the digits,
the GW, UGW, and FGW barycenter contain some artefacts between the digits.
The unbalanced, fused approach can again avoid the unwanted artifacts.

Since the labels of the fused barycenters are here fixed in advanced,
we automatically obtain an embedding of the barycenter such that the numbers are written upright.
In both examples, 
the fused unbalanced Gromov--Wasserstein barycenter provides the best depiction of the numbers 98 and 81.

For the {\bfseries second example}, we take two 2D shapes represented by ($64 \times 64$)-pixel images from the previously mentioned publicly available database \cite{shapes2d}. Both images are of the class ``camel''. 
Similarly to above, we equip each pixel with a unique label. This time our label space $(A,e)$ is given by $A = \{0,1,2,3\}$ with distance
\[
e(a,a') \coloneqq
\begin{cases}
0 &\text{if } a = a'\\
1 &\text{if } \vert a - a' \rvert = 2\\
1/2 &\text{else.}
\end{cases}
\]
We associate the labels $0,1,2,3$ with the head, legs, tail, and hump, respectively.
Notably, the metric $e$ admits the following interpretation:
if two distinct body parts are direct neighbours,
then their distance is $1/2$, otherwise the distance is $1$.
The constructed spaces are shown in the first two columns of \cref{fig:animals_fused}. We proceed as before and compute the barycenters
with respect to the parameters $\rho \coloneqq (0.5, 0.5)$,
$\eps = 0.7 \cdot 10^{-4}$,
$\phi_i \coloneqq 0.01 \, \phi_{\text{BS}}$,
and $\beta = 0.5$.
For the non-fused case, we choose support and distance in $\YY \coloneqq (Y,d,\bullet)$ to be the same underlying ($64 \times 64$)-pixel grid as the inputs. 
For the fused case, we consider the full Cartesian product $Y \times A$, i.e.\
$\Yf \coloneqq (Y \times A,d, e, \bullet)$.
The computed GW, UGW, FGW, and UFGW barycenters are shown in \cref{fig:animals_fused}.

Neither the balanced nor the unbalanced $\GW$ approach finds a meaningful barycenter, i.e.\ an image which illustrates a camel. 
Figuratively, 
these approaches match the head of one input with the hump of the other and vice versa.
This incorrect matching results from large dissimilarity of the single features.
The fused approaches, however, eliminate this issue by exploiting the label information and find a meaningful barycenter.
The visualized transport also shows that in the balanced case two artifacts are present. Notably, the right artifact is orange and thus a result of transporting between the head in one input and the hump in the other. This implies large transport costs which explains why this particular artifact is not present in the unbalanced version while the other artifact is. 

\begin{figure}
    \centering
    \includegraphics[width=0.95\linewidth]{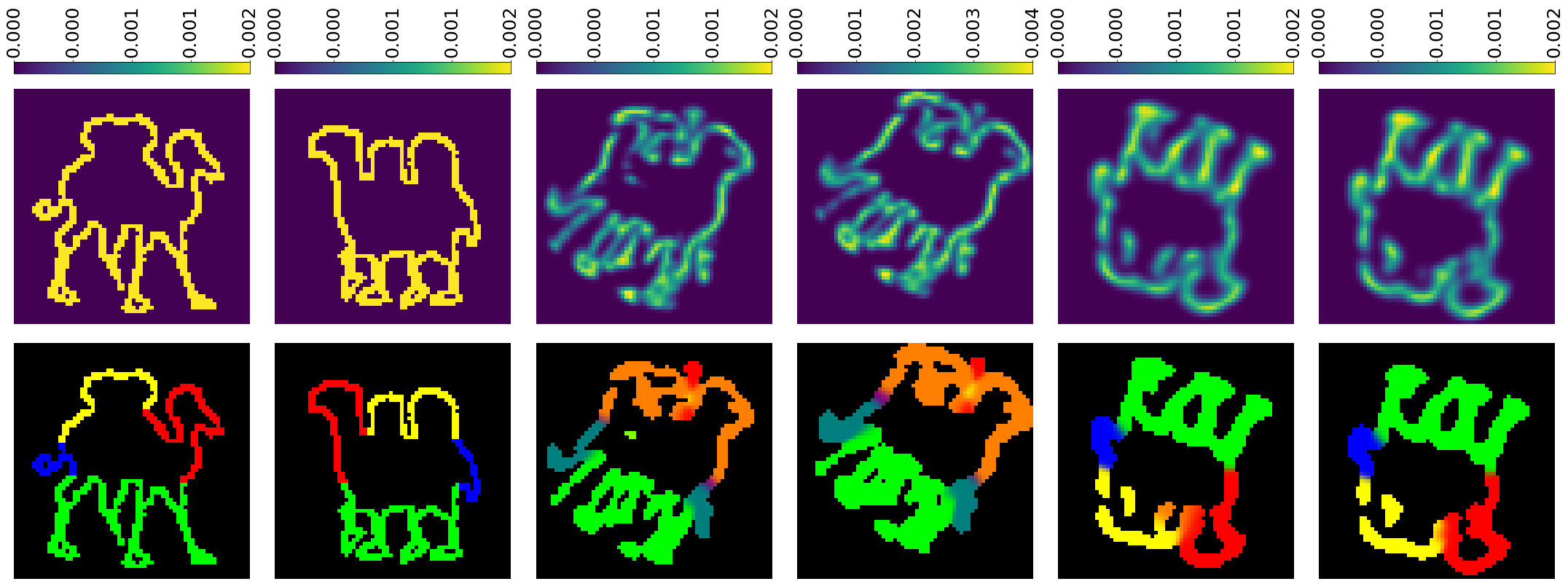}\\
    \scriptsize
    \hspace*{0pt}
    \parbox{\labwidth}{\centering $\XX_1$/$\Xf_1$}
    \parbox{\labwidth}{\centering $\XX_2$/$\Xf_2$}
    \parbox{\labwidth}{\centering GW}
    \parbox{\labwidth}{\centering UGW}
    \parbox{\labwidth}{\centering FGW}
    \parbox{\labwidth}{\centering FUGW}
    \caption{First row, left to right: input marginals, GW barycenter, UGW barycenter, FGW barycenter, UFGW barycenter. 
    Second row, left to right: labels of the input marginals, 
    visualized transport to the barycenters
    by averaging the label colour of the incoming mass.
    }
    \label{fig:animals_fused}
\end{figure}

\begin{figure}
    \centering\small
    \hspace*{27pt}
    $\XX_1$ \hspace{60pt}
    $\XX_2$ \hspace{60pt}
    $\XX_3$ \hspace{60pt}
    $\XX_4$ \hspace{60pt}
    $\XX_5$ 
    \newline
    \rotatebox{90}{
        \footnotesize
        \hspace{-5pt}
        transferred image
        \hspace{15pt}
        clean image
        \hspace{20pt}
        noisy image}
    \includegraphics[width=390pt]{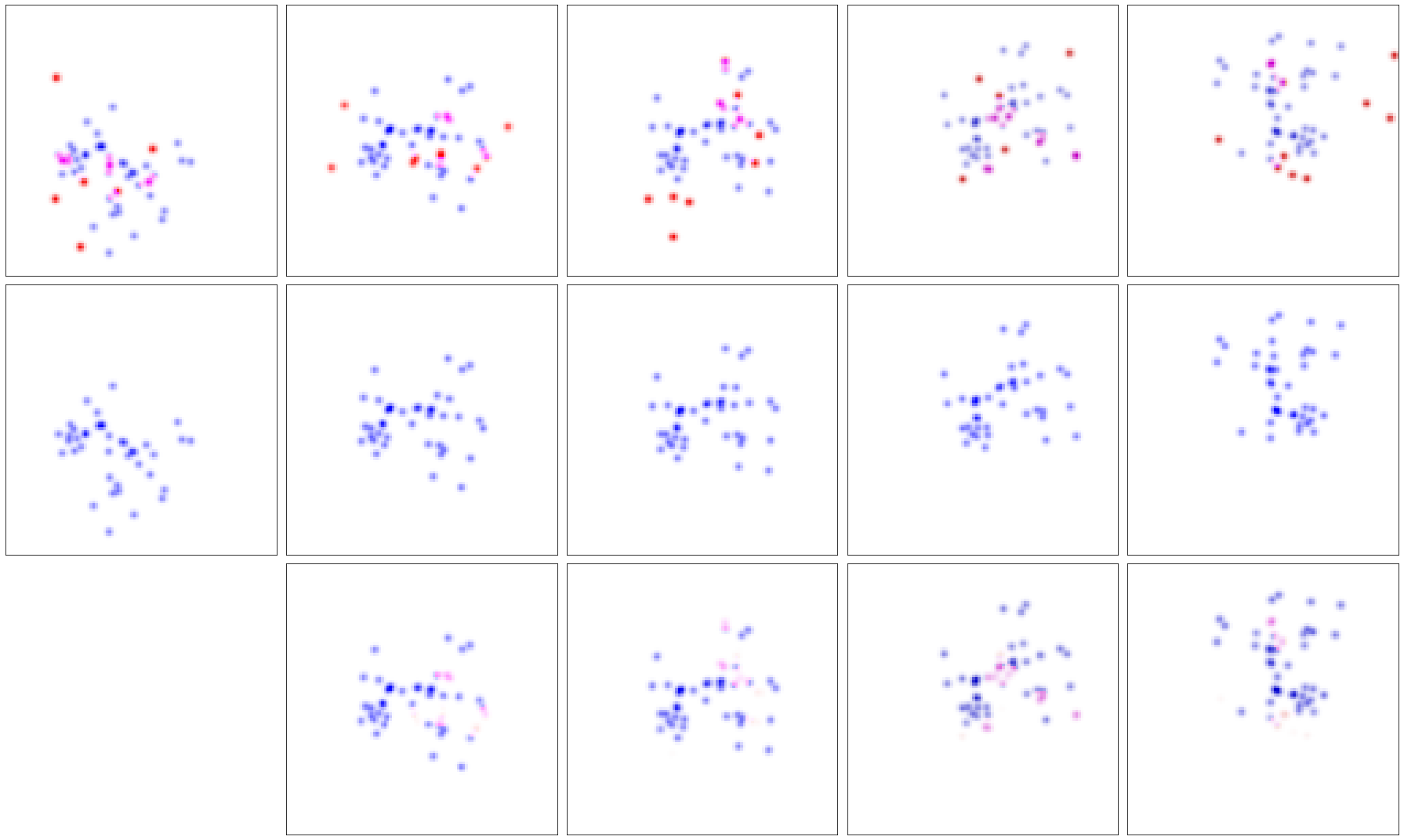}
    \newline
    \hspace*{79pt}
    $\bm K_{1\to2} \mu_1^{\mathrm{clean}}$ \hspace{20pt}
    $\bm K_{1\to3} \mu_1^{\mathrm{clean}}$ \hspace{20pt}
    $\bm K_{1\to4} \mu_1^{\mathrm{clean}}$ \hspace{20pt}
    $\bm K_{1\to5} \mu_1^{\mathrm{clean}}$ 
    \caption{Synthetic snapshots for the numerical experiment in \cref{sec:trans-op}.
    The middle row contains the clean images of the dynamic of 40 coherent particles.
    The top row shows the noisy images, 
    where every images contains 10 additional (noise) particles. 
    The colour encodes clean particles (blau), noise particles (red),
    and the superposition of clean and noise particles (magenta).
    The bottom row shows the transferred first clean image by the computed transfer operators.
    The colour corresponds to the colour in the top row.}
    \label{fig:pd}
\end{figure}

\subsection{Estimating Transfer Operators of Particle Systems}
\label{sec:trans-op}

In the final example, 
multi-marginal GW transport is used 
for the estimation of transfer operators.
Mathematically, a transfer operator
or Frobenius--Perron operator 
is linear map $K \colon L^1(\XX_1) \to L^1(\XX_2)$ 
describing the evolution of a dynamical system 
between two moments in time
\cite{froyland2013analytic,KLNS20}.
Here $L^1(\XX_i)$ for $\XX_i \coloneqq (X_i, d_i, \mu_i)$
denotes the set of all absolutely integrable functions 
on $X_i$ with respect to $\mu_i$. 
The aim of this example is to recover the dynamical evolution
of a coherent particle system from noisy observations.
In this context, \emph{coherent} means that
the relative movement between coherent particles can be neglected.
The synthetic data are generated as follows:
Firstly, 
the positions of 40 coherent particles 
are sampled from the standard Gaussian on $\R^2$.
Slowly rotating the coherent structure and moving it along some curve,
we generate five snapshots of the particle system.
To these snapshots, 
we add 10 further particles serving as noise.
Equipping all particles with a uniform weight,
applying a Gaussian filtering,
and thresholding the resulting image values at $10^{-5}$,
we obtain the synthetic observations in \cref{fig:pd},
where each snapshots consists of $100 \times 100$ pixels
and is interpreted as mm-space $\XX_i$, $i=1,\dots,5$,
like in the previous examples.
The goal is to determine a transfer operator 
that describes the underlying movement
by mapping the density of a single particle in one observation
to the density of the corresponding particle in the next observation.

If the rotation of the system is insignificant,
the transfer operator can be estimated 
using regularized, unbalanced Wasserstein transport \cite{KLNS20,BLNS2021}. 
More precisely,
having the matrix $\mpi$ corresponding to an optimal (regularized, unbalanced) transport plan
between two observations---two measures on a joint, discrete metric space,
we may estimate the transfer operator
by the linear mapping $\bm K$
corresponding to the matrix
\begin{equation}
    \label{eq:trans-op}
    \bm K^{\tT} \coloneqq \diag(\mpi_1)^{-1} \mpi
    \quad\text{with}\quad
    \mpi_1 \coloneqq \mpi^\tT \bm 1,
\end{equation}
where $\bm 1$ denotes the all-ones vector,
and where $0^{-1}$ is set to $0$.
The construction of $\bm K$ can be motivated using statistical physics \cite{KLNS20},
and the quality of $\bm K$ strongly depends on the quality of $\mpi$.
Considering the five snapshots in \cref{fig:pd},
we notice that the particle system is significantly rotating anti-clockwise
such that the Wasserstein transports $\mpi$ and the corresponding $\bm K$
do not match related particles between different observations.

\begin{figure}
    \centering\small
    \includegraphics[width=105pt, clip=true, trim=0pt 0pt 1075pt 0pt, angle=-90]{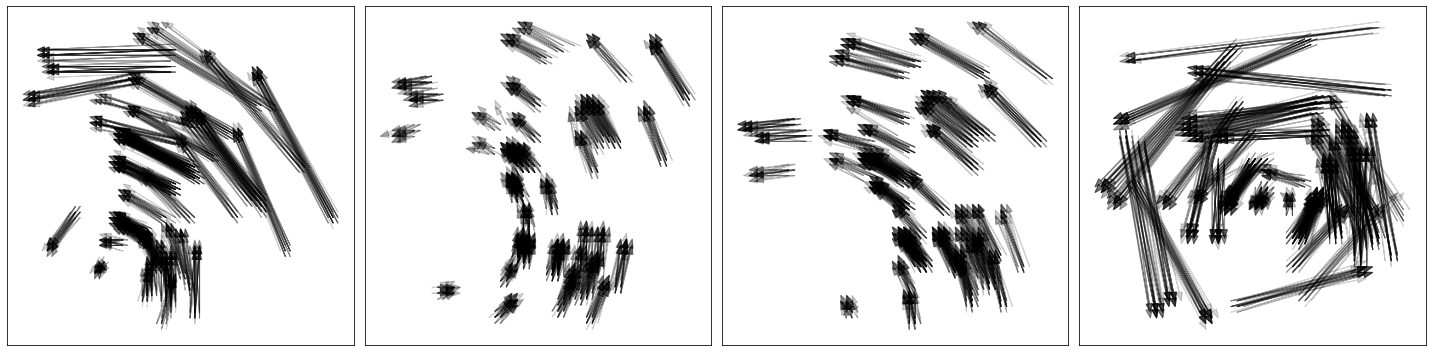}
    \includegraphics[width=105pt, clip=true, trim=359pt 0pt 717pt 0pt, angle=-90]{umgw_transf_ex_2.png}
    \includegraphics[width=105pt, clip=true, trim=717pt 0pt 359pt 0pt, angle=-90]{umgw_transf_ex_2.png}
    \includegraphics[width=105pt, clip=true, trim=1075pt 0pt 0pt 0pt, angle=-90]{umgw_transf_ex_2.png}
    \newline
    $\bm K_{1\to2}$ \hspace{70pt}
    $\bm K_{2\to3}$ \hspace{70pt}
    $\bm K_{3\to4}$ \hspace{70pt}
    $\bm K_{4\to5}$ 
    \caption{The estimated transfer operators between the synthetic snapshots in \cref{fig:pd}.
    The arrows indicate to where the mass of a certain pixel is transferred. 
    The gray values indicate the amount of mass which is transported.
    The figures contain only arrows corresponding to a mass transport larger than $10^{-4}$.}
    \label{fig:pd2}
\end{figure}

Differently from the Wasserstein distance, 
the Gromov--Wasserstein distance is invariant under rotations and translations,
i.e.\ the GW plan between two snapshots $\XX_i$ and $\XX_j$ should match the related particles. 
On the basis of this observation,
we propose to adapt the construction of the transfer operator $\bm K$ in \cref{eq:trans-op}
by replacing $\mpi$ with an optimal GW plan.
To handle the noise particles 
and to estimate all transfer operators simultaneously,
we compute an $\UMGW_\eps$ plan $\mpi_{\UMGW}$ with $\eps \coloneqq 0.2\cdot 10^{-3}$,
where the marginals are unbalanced with respect to $10^{-3} \KL$.
Having the evolution of the system over time in mind,
we employ the cost function
\begin{equation*}
  c \coloneqq \sum_{i=1}^{4} |d_i - d_{i+1}|^2.
\end{equation*}
The transfer operator $\bm K_{i \to j}$ between the snapshots $\XX_i$ and $\XX_j$ is then estimated by
\begin{equation*}
    \bm K^{\tT}_{i \to j} \coloneqq \diag(\mpi_i)^{-1} (P_{X_i \times X_j})_\# \mpi_{\UMGW}
    \quad\text{with}\quad
    \mpi_i \coloneqq (P_{X_i})_\# \mpi_{\UMGW}.
\end{equation*}
The estimated operators $\bm K_{i \to i+1}$, $i=1,\dots,4$,
are shown in \cref{fig:pd2}.
Despite the noise,
the underlying dynamic (transition and rotation) between the snapshots is well recovered.
Moreover,
we use the estimated dynamic to propagate the first clean image $\mu_1^{\mathrm{clean}}$
to $\XX_i$ by computing $\bm K_{1 \to i} \mu_1^{\mathrm{clean}}$ for $i=2,\dots,5$.
In this manner, 
we nearly recover the remaining clean images,
see \cref{fig:pd} bottom row.
Notice that the computed transfer operators do not transfer mass from clean particles
to well separated noise particles.
This numerical example is just a first proof of concept
that multi-marginal, unbalanced, regularized GW transports are able 
to recover dynamics where the Wasserstein transport would fail.
This could be a starting point for further research about GW-based transfer operators.

\section{Conclusions}
We proposed the novel formulation of multi-marginal GW transport and its regularized, unbalanced and fused versions. By establishing the bi-convex relaxation and the barycenter relation, we provide new tools to compute fixed-support GW barycenters. We gave several examples showing that our procedure finds meaningful GW, UGW, FGW, and UFGW barycenters. Furthermore, we provided evidence that our procedure can also be used to create progressive GW interpolations or to denoise particle images.
The employed bi-convex relaxation proved to be tight for appropriate cost functions in the balanced case. 
Up to now, we are not aware of a tightness argument in the unbalanced setting. 
The main problem is here that the marginals of $\pi$ and $\gamma$ do not have to coincide
such that we cannot exploit the marginal conditionally negative definiteness.
Also based on the bi-convex relaxation, 
we may ask what happens if the given marginals of the two underlying plans $\pi$ and $\gamma$ differ in general. This gives rise to a multi-marginal co-optimal transport formulation, whose discussion is left as future work.

%----------------------------------------------------------
\appendix
\section{Proof of \cref{prop:existence}}
%----------------------------------------------------
The proof is based on properties of
Csiszár divergences and the following properties
of the product measure and integral operators.

\begin{lemma}[Product Measure, {\cite[Prop~2.7.8]{bogachev_weak}}]
  \label{lem:product_lsc}
  Let $X$ be a Polish space.
  If $\pi_n \weakly \pi$ converges weakly in $\M^+(X)$,
  then $\pi_n \otimes \pi_n \weakly \pi \otimes \pi$
  converges weakly in $\M^+(X \times X)$.
\end{lemma}

\begin{lemma}[Lower Semi-continuity]   \label{lem:int-lsc}
  Let $c \colon X \times X \to [0,\infty]$ be lower semi-continuous. 
  Then the mapping $\pi \mapsto \int_{X \times X} c(x,x') \dx \pi(x) \dx \pi(x')$
  is weakly lower semi-continuous.
\end{lemma}

\begin{proof}
  For every lower semi-continuous $c$ bounded from below,
  there exists a sequence of (Lipschitz) continuous functions $(c_k)_{k \in \N}$
  with $c_k(x,x') \uparrow c(x,x')$ for all $x,x' \in X$, 
  cf. \cite[p~67]{villani2008optimal}.
  Let $\pi_n \weakly \pi$ be a weakly convergent sequence in $\mathcal M^+(X)$.
  Using \cref{lem:product_lsc},
  we have
  \begin{align*}
    \liminf_{n\to\infty} 
    \int_{X^2} c(x,x') \dx \pi_n(x) \dx \pi_n(x')
    &\ge
      \liminf_{n\to\infty} 
      \int_{X^2} c_k(x,x') \dx \pi_n(x) \dx \pi_n(x')
    \\
    &= 
      \int_{X^2} c_k(x,x') \dx \pi(x) \dx \pi(x')
  \end{align*}
  for all $k \in \N$.  
  Taking the supremum over $k \in \N$
  and applying Lebesgue's dominated convergence theorem
  establishes the weak lower semi-continuity.
\end{proof}

\begin{proof}[Proof of \cref{prop:existence}]
  The objective $F_\varepsilon$ in $\UMGW_\epsilon$ is weakly lower semi-continuous 
  as sum of weakly lower semi-continuous functions.
  More precisely, 
  the integral is weakly lower semi-continuous by \cref{lem:int-lsc}. 
  The remaining terms of the sum are compositions 
  of weakly lower semi-continuous divergences
  and the weakly continuous mappings 
  $\pi \mapsto \pi \otimes \pi$,
  see \cref{lem:product_lsc},
  and
  $\pi \mapsto (P_{X_i})_\# \pi = \pi_i$.
  By Jensen's inequality, we know that
	for every  $\mu, \nu \in \M^+(\mathcal X)$ with $\| \nu \|_{\TV} > 0$,
  it holds
  \begin{equation*}
    D_\phi(\mu, \nu)
    \ge
    \| \nu \|_{\TV} \, 
    \phi( \| \mu \|_{\TV} / \| \nu \|_{\TV}) ,
  \end{equation*}
	see \cite[(2.44)]{LMS18}.
  Therefore, we may bound the objective from below by
  \begin{align*}
    \F_\varepsilon(\pi)
    &\geq 
      \sum_{i=1}^N 
      D_{\phi_i}^\otimes(\pi_i,\mu_i)
      +  \varepsilon \KL^\otimes(\pi,\nu^\otimes)
    \\
    &\geq 
      \sum_{i=1}^N 
      \| \mu_i \|_{\TV}^2 \,
      \phi_i\Bigl(\tfrac{\|\pi_i\|_{\TV}^2}{\|\mu_i\|_{\TV}^2}\Bigr)
      + \varepsilon \, \|\nu^\otimes\|_{\TV}^2 \,
      \phi_{\KL} \Bigl( \tfrac{\|\pi\|_{\TV}^2}{\|\nu^\otimes\|_{\TV}^2} \Bigr)
    \\
    &= \|\pi\|_{\TV}^2 \,
      \biggl[
      \sum_{i=1}^N 
      \underbrace{
      \tfrac{\|\mu_i\|_{\TV}^2}{\|\pi\|_{\TV}^2} \, 
      \phi_i \Bigl(\tfrac{\|\pi\|_{\TV}^2}{\|\mu_i\|_{\TV}^2}\Bigr)
      }_{\to (\phi_i)'_\infty}
      +
      \varepsilon \,
      \underbrace{
      \tfrac{\|\nu^\otimes\|_{\TV}^2}{\|\pi\|_{\TV}^2} \, 
      \phi_{\KL} \Bigl(\tfrac{\|\pi\|_{\TV}^2}{\|\nu^\otimes\|_{\TV}^2}\Bigr)
      }_{\to \infty}
      \biggr].
  \end{align*}
  For both assumptions,
  $F_\varepsilon(\pi)$ diverges for $\|\pi\| \to \infty$
  showing coercivity.
  Consequently,
  every minimizing sequence $(\pi^{(n)})_{n \in \N} \in \M^+(X)$ 
  with $F_\varepsilon(\pi^{(n)}) \downarrow \UMGW_\varepsilon (\XX_1, \dots, \XX_N)$ 
  is uniformly bounded in total variation.
  The theorem of Banach--Alaoglu guarantees 
  the existence of a weakly convergent subsequence 
  $\pi^{(k)} \weakly \tilde \pi \in \M^+(X)$.
  The weak lower semi-continuity of $F_\varepsilon$ ensures
  that $\tilde \pi$ is a minimizer.
\end{proof}

\section*{Acknowledgments}
The authors thank the anonymous reviewers for their valuable suggestions. This work is supported in part by funds from the German Research Foundation (DFG) within the RTG
2433 DAEDALUS and by the BMBF project ``VI-Screen'' (13N15754).

\bibliographystyle{abbrv}
\bibliography{reference}

\end{document}